\documentclass[11pt,twoside]{article}

\usepackage{amsmath, amsthm,amsfonts,amssymb,mathtools,mathdots,scalerel,mathrsfs,stmaryrd,cmll}
\usepackage[utf8]{inputenc}
\usepackage[T1]{fontenc}
\usepackage{mlmodern, tikz-cd, quiver}
\usepackage[many]{tcolorbox}

\usepackage{enumitem}
\usepackage[hidelinks]{hyperref}
\usepackage{tikz, fancyhdr, xparse, xcolor}
\usepackage[british]{babel}
\usepackage[a4paper, top=4.5cm, bottom=4.5cm, left=4cm, right=4cm, asymmetric]{geometry}

\usepackage{crossreftools}

\usepackage{ahtitle, titlesec}
\usepackage{etoolbox}
\makeatletter
\patchcmd{\ttlh@hang}{\parindent\z@}{\parindent\z@\leavevmode}{}{}
\patchcmd{\ttlh@hang}{\noindent}{}{}{}
\makeatother

\newcommand\eqdef{\coloneqq}
\newcommand\nbd{\nobreakdash-\hspace{0pt}}
\newcommand\idd[1]{\mathrm{id}_{#1}}
\newcommand\bigid[1]{\mathrm{Id}_{#1}}
\newcommand\invrs[1]{#1^{-1}}

\newcommand\incl{\hookrightarrow}

\newcommand\surj{\twoheadrightarrow}
\newcommand\restr[2]{{#1}{\raisebox{0pt}{$|_{#2}$}}}

\newcommand\set[1]{\left\{ {#1} \right\}}

\newcommand\slice[2]{{#1}/{\raisebox{-2pt}{$#2$}}}
\newcommand\join{\,{\star}\,}

\newcommand\cat[1]{\mathbf{#1}}

\newcommand\fun[1]{\mathsf{#1}}

\newcommand\ogpos{\cat{ogPos}}

\newcommand\poscat{\cat{Pos}}

\newcommand\thetacat{\Theta}
\newcommand\simplexcat{\Delta}

\newcommand\rdcpx{\cat{RDCpx}}
\newcommand\rdcpxmap{\cat{RDCpx}_{\downarrow}}

\newcommand\sset{\cat{sSet}}

\newcommand\hasse[1]{\mathscr{H}{#1}}

\newcommand\clset[1]{\mathrm{cl}\set{#1}}

\newcommand\bound[2]{\partial_{#1}^{#2}}
\newcommand\faces[2]{\Delta_{#1}^{#2}}
\newcommand\cofaces[2]{\nabla_{#1}^{#2}}

\newcommand\cp[1]{\,{\scriptstyle\#}_{#1}\,}

\newcommand\celto{\Rightarrow}

\newcommand\submol{\sqsubseteq}

\newcommand\gray{\otimes}
\newcommand\augm[1]{{{#1}_\bot}}
\newcommand\dimin[1]{{#1}_{\not\bot}}

\newcommand\cubeconn[2]{\gamma^{#1}_{#2}}

\newcommand\thearrow[1]{{#1}\vec{I}}

\newcommand{\C}{\mathscr{C}}
\newcommand{\Ob}{\mathrm{Ob}}
\renewcommand{\a}{\alpha}
\renewcommand{\b}{\beta}
\newcommand{\bd}{\partial}
\newcommand{\atom}{{\scalebox{1.3}{\( \odot \)}}}
\newcommand{\smallatom}{\odot}
\newcommand{\Set}{\cat{Set}}
\newcommand{\sSet}{\cat{sSet}}
\newcommand{\Cat}{\cat{Cat}}
\newcommand{\pt}{\mathbf{1}}

\newcommand{\Sd}{\mathrm{Sd}}
\newcommand{\Ex}{\mathrm{Ex}}

\newcommand{\An}{\mathrm{An}}

\renewcommand{\rdcpx}{\cat{RDCpx}}
\DeclareMathOperator*{\colim}{colim}
\DeclareMathOperator*{\Nd}{Nd}

\newcommand{\W}{\mathcal{W}}

\newcommand{\ptcat}{\mathbf{1}}

\newtheoremstyle{ittheorem}
  {\topsep}   % ABOVESPACE
  {\topsep}   % BELOWSPACE
  {\itshape}  % BODYFONT
  {0pt}       % INDENT (empty value is the same as 0pt)
  {\sffamily \itshape \bfseries} % HEADFONT
  { ---}         % HEADPUNCT
  {5pt plus 1pt minus 1pt} % HEADSPACE
  {}          % CUSTOM-HEAD-SPEC

\newtheoremstyle{itdfn}
  {\topsep}   % ABOVESPACE
  {\topsep}   % BELOWSPACE
  {}  % BODYFONT
  {0pt}       % INDENT (empty value is the same as 0pt)
  {\sffamily \itshape \bfseries} % HEADFONT
  {}         % HEADPUNCT
  {5pt plus 1pt minus 1pt} % HEADSPACE
  {\thmnumber{#2}{\thmnote{\normalfont\ \ %
  {\sffamily(#3)}.}}}          % CUSTOM-HEAD-SPEC

\newtheoremstyle{itrmk}
  {0.5\topsep}   % ABOVESPACE
  {0.5\topsep}   % BELOWSPACE
  {\normalfont}  % BODYFONT
  {0pt}       % INDENT (empty value is the same as 0pt)
  {\sffamily \itshape} % HEADFONT
  { ---}         % HEADPUNCT
  {5pt plus 1pt minus 1pt} % HEADSPACE
  {}          % CUSTOM-HEAD-SPEC
  
\newtheoremstyle{itexm}
  {0.5\topsep}      % ABOVESPACE
  {0.5\topsep}      % BELOWSPACE
  {\normalfont}     % BODYFONT
  {0pt}             % INDENT (empty value is the same as 0pt)
  {\sffamily \itshape \bfseries \color{\mycolor}}        % HEADFONT
  {\\}               % HEADPUNCT
  {5pt plus 1pt minus 1pt} % HEADSPACE
  {\thmname{#1} \thmnumber{#2}{\thmnote{\normalfont\ \ %
  {\sffamily(#3)}.}}}           % CUSTOM-HEAD-SPEC
  
\makeatletter
  \renewcommand\@upn{\textit}
\makeatother

\theoremstyle{ittheorem}
\newtheorem{thm}{Theorem}[section]
\newtheorem*{thm*}{Theorem}
\newtheorem{prop}[thm]{Proposition}
\newtheorem*{prop*}{Proposition}
\newtheorem{cor}[thm]{Corollary}
\newtheorem{lem}[thm]{Lemma}

\theoremstyle{itdfn}
\newtheorem{dfn}[thm]{}
\theoremstyle{itrmk}
\newtheorem{rmk}[thm]{Remark}
\newtheorem{comm}[thm]{Comment}

\setlist{leftmargin=20pt,parsep=0pt,itemsep=0pt,topsep=1ex}

\titleformat{\section}
 {\large\scshape}{\thesection.}{1em}{}

\titleformat{\subsection}
 {\normalsize\itshape}{\thesubsection.}{1em}{}
\titlespacing*{\subsection}
{0pt}{1.5ex plus 1ex minus .2ex}{1.5ex plus .2ex}

\newcommand\runtitle{diagrammatic sets as a model of homotopy types}
\newcommand\runauthor{chanavat and hadzihasanovic}

\title{Diagrammatic sets as a model of homotopy types}

\author{Cl\'emence Chanavat and Amar Hadzihasanovic}

\institution{Tallinn University of Technology}

\begin{document}

\thispagestyle{empty}
\maketitle 

\noindent\makebox[\textwidth][r]{%
	\begin{minipage}[t]{.7\textwidth}
\small \emph{Abstract.}
	Diagrammatic sets are presheaves on a rich category of shapes, whose definition is motivated by combinatorial topology and higher-dimensional diagram rewriting.
	These shapes include representatives of oriented simplices, cubes, and positive opetopes, and are stable under operations including Gray products, joins, suspensions, and duals.
	We exhibit a cofibrantly generated model structure on diagrammatic sets, as well as two separate Quillen equivalences with the classical model structure on simplicial sets.
	We construct explicit sets of generating cofibrations and acyclic cofibrations, and prove that the model structure is monoidal with the Gray product of diagrammatic sets.
\end{minipage}}

\vspace{20pt}

\makeaftertitle

\normalsize

\section*{Introduction}

After Quillen's foundational article \cite{quillen_homotopical_1967}, the general method to present a homotopy category is to endow a category with a model structure.
The first and foremost examples are the Kan--Quillen model structure on simplicial sets and the Quillen model structure on topological spaces, which are equivalent via the pair of geometric realisation and its right adjoint nerve functor, and both present the category of homotopy types of CW\nbd complexes.

Simplicial sets are among several combinatorial models of spaces defined as presheaves over a \emph{shape category}, whose objects typically model ``combinatorial \( n \)\nbd balls'' for each \( n \geq 0 \), and morphisms model inclusions into their spherical boundaries, as well as the attaching maps used in the construction of cell complexes.
In \textit{Pursuing Stacks} \cite{grothendieck2021pursuing}, Grothendieck conjectured that presheaves on a certain class of small categories, called \emph{test categories}, would naturally model the homotopy types of CW-complexes.
After the developments of Maltsiniotis \cite{Maltsiniotis_homotopy} and, in particular, Cisinski \cite{cisinski_prefaisceaux_2006}, we know this to be true: 
every test category gives rise to a model structure together with a canonical Quillen equivalence with the classical model structure on simplicial sets.  
Many common categories of shapes have since been proven to be test categories, including the category of simplices \cite{cisinski_prefaisceaux_2006, Cisinski2011_theta}, varieties of cubical categories \cite{Maltsiniotis2009, buchholtz2017varieties,cisinski_prefaisceaux_2006}, Joyal's \( \thetacat \) \cite{Cisinski2011_theta, Ara_2011_theta}, the category of dendroids \cite{ARA2018_dendroid}, and the category of positive opetopes \cite{zawadowski2018positive}.

Among these, simplicial and cubical shapes have been overwhelmingly popular in homotopy theory, due to their simplicity and uniformity.
These, however, may run counter to the interest of combinatorial and computational topologists in having ``small'' cellular models of certain spaces, which typically require more complex polytopal shapes.
The situation is even more serious in the theory of higher categories, where shape categories for homotopy types have often been repurposed for the modelling of higher-categorical structures, typically by assigning each shape a particular \emph{direction} or orientation: examples include the complicial model of \( (\infty, n) \)\nbd categories \cite{verity2008complicial} based on simplices, the comical model \cite{campion2020cubical} based on cubes, and Rezk's model \cite{rezk2010cartesian} based on \( \thetacat \).
Here, ``directed cell complexes'' show up naturally as presentations of higher-dimensional algebraic theories or rewrite systems \cite{guiraud2019rewriting, ara2023polygraphs}, which in these models may not admit a representation with equivalent computational properties, preventing a satisfactory account of their functorial semantics.
Furthermore, important higher-categorical constructions such as Gray products, duals, suspensions, and joins \cite{ara2020joint}, which are most easily described in terms of cellular structure, do not admit simple explicit models when the shape category is not closed under each construction.

Building on Steiner's work \cite{steiner_algebra_1993}, as well as his own \cite{hadzihasanovic_combinatorial-topological_2019, hadzihasanovic_higher-dimensional_2023}, the second-named author has recently laid out in \cite{amar_pasting} a combinatorial framework for a theory of ``regular directed cell complexes'' as applicable to higher-categorical diagram rewriting.
In this framework, a \emph{regular directed complex} is encoded by the face poset of its underlying cell complex, complemented with orientation data partitioning the set of faces of each cell into an \emph{input} and an \emph{output} half.
The definition of regular directed complex rests on the definition of a particular inductive subclass, the \emph{molecules}, encoding well-formed shapes of \( n \)\nbd categorical pasting diagrams.
An \emph{atom}, which is a molecule with a greatest face, can be seen as a combinatorial model of a directed ball: indeed, it is provable that the underlying poset of each atom is the face poset of a regular CW\nbd ball, and the underlying poset of a regular directed complex is the face poset of a regular CW\nbd complex.

Atoms have the property of being closed under Gray products, duals, suspensions, and joins, and include exemplars of many of the aforementioned shapes, including oriented simplices, cubes, and positive opetopes.
They appear, thus, to be an advantageous class of shapes for higher categories, especially for the purposes of higher-dimensional algebra and rewriting.
In this article, we start our exploration of atoms and regular directed complexes in the context of the homotopy theory of higher categories, by first settling the case of higher groupoids also known as homotopy types.

We define a notion of morphism between regular directed complexes, that we call a \emph{cartesian map}, such that the category \( \atom \) (\emph{atom}) of atoms and cartesian maps is Eilenberg--Zilber.
Using the theory developed in \cite{cisinski_prefaisceaux_2006,cisinski_higher_2019}, as well as the closure of atoms under Gray products providing representable functorial cylinders, it becomes almost a formality to prove that \( \atom \) is a test category, and we obtain many tools for better characterising the induced model structure.

We call a presheaf on \( \atom \) a \emph{diagrammatic set}.
The name is borrowed from Kapranov and Voevodsky \cite{Kapranov1991}, who similarly considered presheaves on a category of ``combinatorial pasting diagrams'' --- in their case, it was Johnson's composable pasting schemes \cite{johnson1989combinatorics} --- as a model of homotopy types.
For reasons discussed in \cite{henry2019non}, this attempt was flawed, and it is clear that the authors did not have a strong grip on the combinatorics of either the shapes or their morphisms.
Nevertheless, the results of this article may be seen as a vindication of their general idea.
We note that this article follows, and subsumes in part an earlier attempt by the second-named author to resurrect diagrammatic sets \cite{hadzihasanovic_diagrammatic_2020}, which used a broader notion of morphism of atoms (resulting in a shape category that was not Eilenberg--Zilber), and only resulted in the proof of a weaker version of the homotopy hypothesis.

Our main theorem is as follows.
\begin{thm*}
    There is a cofibrantly generated model structure on \( \atom\Set \), whose 
    \begin{itemize}
	    \item cofibrations are the monomorphisms, and
	 \item acyclic cofibrations are generated by ``horn inclusions'' into atoms.
    \end{itemize}
    This model structure is Quillen-equivalent to the classical model structure on simplicial sets, both via a left Quillen ``subdivision'' functor, and a right Quillen ``restriction to simplices'' functor.
    
    \noindent Moreover, the model structure is monoidal with respect to the Gray product, which induces the cartesian product in the homotopy category.
\end{thm*}
\noindent We refer the reader to Section \ref{sec:homotopy} for more precise statements.

\subsection*{Structure of the article}

The article is divided into 3 sections.
Section \ref{sec:atom} can be seen as a direct continuation of \cite[Section 6.2]{amar_pasting}, which studied maps of regular directed complexes.
Restricting to the category \( \rdcpx \) of maps whose underlying order-preserving map is, in particular, a Grothendieck fibration of posets, we obtain a shape category \( \atom \) that is Eilenberg--Zilber (Proposition \ref{prop:atom_ez_category}).
Then, we show that the Gray product and join operations on regular directed complexes determine monoidal structures on both \( \rdcpx \) and \( \atom \).

In Section \ref{sec:diagrammatic}, we define diagrammatic sets as presheaves over \( \atom \), whose category we denote by \( \atom\Set \).
Then, we fully and faithfully embed \( \rdcpx \) into \( \atom\Set \), factorising the Yoneda embedding as
\begin{equation*}
    \atom \incl \rdcpx \incl \atom\Set.
\end{equation*}
We also show that colimits of inclusions computed in \( \rdcpx \) can be identified with the same colimits computed in \( \atom\Set \), and characterise the regular directed complexes as the \emph{regular} presheaves, where the classifying morphism of each element is a monomorphism
(Proposition \ref{prop:rdcpx_are_regular}).

In Section \ref{sec:homotopy}, we deduce that the category \( \atom \) is a test category. 
Thus, by the theory developed in \cite{cisinski_prefaisceaux_2006}, we obtain immediately a model structure on \( \atom\Set \), which we call the \emph{Cisinski model structure} and spend the rest of the article studying.
We show that the left Kan extension of the functor taking the simplicial nerve of the underlying poset of an atom is a Quillen equivalence with the classical model structure on simplicial sets.
Then, we extend the Gray product to diagrammatic sets via Day convolution, and show that ``Gray product with the arrow'' is a good functorial cylinder (Proposition \ref{prop:exact_cylinder_object}).
We define a set of horn inclusions, generalising the simplicial horn inclusions, and we show that they generate a class of anodyne extensions with respect to our cylinder (Proposition \ref{prop:lr_J_is_anodyne}).
Thus, using again \cite{cisinski_prefaisceaux_2006}, we are able to show that the horn inclusions are a generating set of acyclic cofibrations (Theorem \ref{thm:main_theorem}), and that the model structure is monoidal with respect to the Gray product (Theorem \ref{thm:model_structure_monoidal_Gray}).
We conclude by proving that \( \atom \) is in fact a strict test category
(Proposition \ref{prop:map_gray_to_cart_is_weak_eq}), and, in particular, the cartesian product and the Gray product of diagrammatic sets coincide in the homotopy category.

\subsection*{Notation and prerequisites}

We recall some basic facts about test categories, and refer to \cite{cisinski_prefaisceaux_2006, cisinski_higher_2019} for details.
We know from Thomason \cite{thomason_model} that \( \Cat \), the category of small categories, is a model for the homotopy theory of CW-complexes. 
The class of weak equivalences of this model structure is called \( \mathcal{W}_\infty \), and is an instance of what Grothendieck calls in \emph{Pursuing Stacks} a \emph{basic localiser}: a class of functors in \( \Cat \) one wishes to invert in order to model various homotopical phenomena. 
The class \( \mathcal{W}_\infty \) is, in fact, the smallest such basic localiser. 

The relationship between basic localisers and test categories is as follows.
Given a small category \( \C \), one gets a canonical functor \( i_\C \colon \C \to \Cat \) sending an object \( c \) to the slice \( \slice{\C}{c} \). 
By left Kan extension along the Yoneda embedding into the presheaf category \( \widehat \C \), this produces an adjunction
\[\begin{tikzcd}
	{\widehat\C} && \Cat
	\arrow[""{name=0, anchor=center, inner sep=0}, "{\int_\C}", curve={height=-12pt}, from=1-1, to=1-3]
	\arrow[""{name=1, anchor=center, inner sep=0}, "{N_\C}", curve={height=-12pt}, from=1-3, to=1-1]
	\arrow["\dashv"{anchor=center, rotate=-90}, draw=none, from=0, to=1]
\end{tikzcd}\]
where \( \int_\C \) sends a presheaf to its category of elements, and \( N_\C \) is its right adjoint nerve functor.
Given a basic localiser \( \mathcal{W} \) on \( \Cat \), we get a class of maps \( \mathcal{W}_{\C} \eqdef \invrs{\int_\C}(\mathcal{W}) \) in \( \widehat\C \), and we say that \( \C \) is a \emph{\( \mathcal{W} \)\nbd weak test category} if the adjunction induces an equivalence between the localised categories
\[\begin{tikzcd}
	{\widehat\C[\invrs{\mathcal{W}_\C}]} && \Cat[\invrs{\mathcal{W}}].
	\arrow[""{name=0, anchor=center, inner sep=0}, "{\sim}", curve={height=-12pt}, from=1-1, to=1-3]
	\arrow[""{name=1, anchor=center, inner sep=0}, "{\sim}", curve={height=-12pt}, from=1-3, to=1-1]
	\arrow["\dashv"{anchor=center, rotate=-90}, draw=none, from=0, to=1]
\end{tikzcd}\]
We say that \( \C \) is a \emph{\( \mathcal{W} \)\nbd test category} if each slice \( \slice{\C}{c} \) is a \( \mathcal{W} \)\nbd weak test category, and \( \C \) is \( \mathcal{W} \)\nbd contractible in the sense that \( \C \to \ptcat \) belongs to \( \mathcal{W} \).
In this article, we are only concerned with the basic localiser \( \mathcal{W}_\infty \), and, as customary in this case, we drop the prefix \( \mathcal{\W}_\infty \) and simply say ``test category''.

With regard to regular directed complexes, we reuse the concepts and notations introduced in \cite{amar_pasting}. 
Ideally, the reader would be familiar with the first 3 chapters; nevertheless, we give a brief overview.
Given a poset \( P \), the \emph{covering relation} \( \succ \) is defined by \( y \succ x \) if and only if \( y > x \), and any element \( z \) between \( x \) and \( y \) is either equal to \( x \) or \( y \).
In this case, we say that \( y \) \emph{covers} \( x \).
We write \( \prec \) for the converse of \( \succ \).
The relation \( \succ \) defines a directed graph \( \hasse{P} \), the \emph{covering diagram}, also known as Hasse diagram.
When the poset has \emph{locally finite height}, in the sense that, for all \( x \in P \), any chain in the lower set of \( x \) is finite, then \( P \) can be reconstructed from \( \hasse{P} \).
We say that a poset \( P \) of locally finite height is \emph{graded} if, for all \( x \in P \), all maximal paths starting from \( x \) in \( \hasse{P} \) have the same length. 
This defines a function \( \dim \colon P \to \mathbb N \), the \emph{dimension}, assigning to each element the length of a maximal path starting from it. 
The \emph{dimension} of a graded poset is \( -1 \) if the poset is empty, the maximum of the dimensions of its elements if it exists, and \( \infty \) otherwise.

An \emph{orientation} on a graded poset is the assignment of a value \( \a \in \{ -, + \} \) to each edge of its covering diagram. 
If \( x \prec y \) with orientation \( - \), then \( x \) is in an \emph{input face} of \( y \), while if the orientation is \( + \), \( x \) is an \emph{output face} of \( y \). 
We denote by \( \Delta^- y \) and \( \Delta^+ y\), respectively, the sets of input and output faces of \( y \). 
An \emph{oriented graded poset} is a graded poset with an orientation.
A morphism \( f\colon P \to Q \) of oriented graded posets is a function which induces bijections between \( \faces{}{\a}x \) and \( \faces{}{\a}f(x) \) for all \( x \in P \) and \( \a \in \set{ +, - } \).
With their morphisms, oriented graded posets form a category \( \ogpos \).

Oriented graded posets can be used to encode the shape of \( n \)-categorical pasting diagrams. 
To convince the reader, we give an example of a pasting diagram on the left and its oriented face poset on the right:
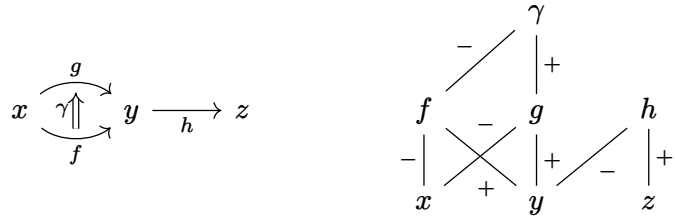
\[    
\begin{tikzcd}
        &&&&& \gamma \\
        x & y & z && f & g & h \\
        &&&& x & y & z
	\arrow["{-}"', no head, from=1-6, to=2-5]
        \arrow["{+}", no head, from=1-6, to=2-6]
        \arrow[""{name=0, anchor=center, inner sep=0}, "f"', curve={height=12pt}, from=2-1, to=2-2]
        \arrow[""{name=1, anchor=center, inner sep=0}, "g", curve={height=-12pt}, from=2-1, to=2-2]
        \arrow["h"', from=2-2, to=2-3]
        \arrow["{-}"', no head, from=2-5, to=3-5]
        \arrow["{+}", no head, from=2-6, to=3-6]
        \arrow["{-}", no head, from=2-7, to=3-6]
        \arrow["{+}", no head, from=2-7, to=3-7]
        \arrow["{-}"{pos=0.8}, no head, from=3-5, to=2-6]
        \arrow["{+}"{pos=0.2}, no head, from=3-6, to=2-5]
        \arrow["\gamma", shorten <=3pt, shorten >=3pt, Rightarrow, from=0, to=1]
    \end{tikzcd}
\]
This is a \( 2 \)\nbd dimensional oriented graded poset, with maximal elements \( \gamma \) of dimension 2 and \( h \) of dimension 1.

We say that \( U \subseteq P \) is \emph{closed} if it contains the lower set of each of its elements; these are the usual closed sets of the Alexandrov topology on a poset.
Given a closed subset \( U \), we can define its input (output) \( n \)\nbd boundary \( \bd^-_n U \) (\( \bd^+_n U \)) as the closure of the \( n \)\nbd dimensional elements of \( U \) that are output (input) faces of no \( (n + 1) \)\nbd elements of \( P \), as well as of the maximal \( k \)\nbd dimensional elements of \( U \) for \( k < n \).
We omit the index \( n \) when it is equal to \( \dim{U} - 1 \).
Given \( x \in P \), we write \( \bd^\a_n x \) for \( \bd^\a_n \clset{x} \).
For instance, if \( U = P \) as in the previous example, then \( \bd^-_1 U \) and \( \bd^+_1 U \) are, respectively,
\[\begin{tikzcd}
	x & y & z && x & y & z.
	\arrow["f"', curve={height=12pt}, from=1-1, to=1-2]
	\arrow["h"', from=1-2, to=1-3]
	\arrow["g", curve={height=-12pt}, from=1-5, to=1-6]
	\arrow["h"', from=1-6, to=1-7]
\end{tikzcd}\]
An oriented graded poset \( P \) is \emph{globular} if, for all \( n > k \geq 0 \) and \( \a, \b \in \set{ +, - } \), we have \( \bound{k}{\a} \bound{n}{\b} P = \bound{k}{\a} P \).
It is \emph{round} if, furthermore, for all \( k < \dim P \), \( \bd^-_k P \cap \bd^+_k P = \bd_{k - 1} P \), where \( \bd_n P \eqdef \bd^-_n P \cup \bd^+_n P \). 

The \emph{molecules} are an inductive subclass of the oriented graded posets, closed under isomorphism, and defined by three constructors.
First of all, the \emph{point} \( \pt \), which is the poset with one element and trivial orientation, is a molecule. 
If \( U, V \) are molecules, and if \( \bd^+_k U \) and \( \bd^-_k V \) are isomorphic, then the \emph{pasting of \( U \) and \( V \) at the \( k \)\nbd boundary} \( U \cp k V \), obtained as the pushout
\[    
	\begin{tikzcd}[sep=small]
        {\bd^+_k U \cong \bd^-_k V} & V \\
        U & {U \cp k V}
        \arrow[from=1-1, to=1-2]
        \arrow[from=1-1, to=2-1]
        \arrow[from=1-2, to=2-2]
        \arrow[from=2-1, to=2-2]
	\arrow["\lrcorner"{anchor=center, pos=0.125, rotate=180}, draw=none, from=2-2, to=1-1]
    \end{tikzcd}
\]
in \( \ogpos \), is a molecule.
Finally, if \( U, V \) are round molecules of the same dimension, such that \( \bd U \) and \( \bd V \) are isomorphic, and such that this isomorphism restricts to isomorphisms \( \bd^\a U \cong \bd^\a V \), then we may construct the pushout
\[
	\begin{tikzcd}[sep=small]
        {\bd U \cong \bd V} & V \\
        U & {\bd(U \celto V)}
        \arrow[from=1-1, to=1-2]
        \arrow[from=1-1, to=2-1]
        \arrow[from=1-2, to=2-2]
        \arrow[from=2-1, to=2-2]
	\arrow["\lrcorner"{anchor=center, pos=0.125, rotate=180}, draw=none, from=2-2, to=1-1]
    \end{tikzcd}
\]
in \( \ogpos \).
Then the oriented graded poset \( U \celto V \), obtained by adjoining a greatest element \( \top \) to \( \bd (U \celto V) \) with \( \bd^- \top = U \) and \( \bd^+ \top = V \), is a molecule.
An \emph{atom} is a molecule with a greatest element; it can be shown that an atom is either \( \pt \), or is of the form \( U \celto V \).
A \emph{regular directed complex} is an oriented graded poset such that the lower set of each of its elements is an atom. 

The following summarises some useful properties of molecules.
\begin{prop*}    
Let \( U \) be a molecule. 
Then
    \begin{enumerate}
	 \item \( U \) is globular,
	 \item \( U \) is rigid, that is, it has no non-trivial automorphisms,
	\item if \( U \) is an atom, then it is round,
        \item for all \( \a \in \{ -, + \}, n \geq 0 \), \( \bd^\a_n U \) is a molecule,
        \item \( U \) is a regular directed complex.
    \end{enumerate}
\end{prop*}
\noindent A consequence of the rigidity of molecules is that \( U \cp{k} V \) and \( U \celto V \) are independent of the choice of isomorphism between the boundaries of \( U \) and \( V \), as implied by the notation.
A further useful property of regular directed complexes is \emph{oriented thinness}:
this says that, for all \( x < y \) with \( \dim x + 2 = \dim y \), the interval \( [x, y] \) is of the form
\[        
	\begin{tikzcd}[sep=small]
            & y \\
            {t_1} && {t_2} \\
            & x
            \arrow["\a"', no head, from=1-2, to=2-1]
            \arrow["\gamma", no head, from=1-2, to=2-3]
            \arrow["\beta"', no head, from=2-1, to=3-2]
            \arrow["\delta"', no head, from=3-2, to=2-3]
\end{tikzcd}
\]
with the orientations satisfying \( \a \beta = - \gamma \delta \).

The \emph{submolecule inclusions} are the smallest class of morphisms of molecules closed under composition, and containing the isomorphisms and the canonical inclusions \(U \incl U \cp k V \) and \( V \incl U \cp k V \).
If \( V \subseteq U \), we write \( V \submol U \), and say that \( V \) is a submolecule of \( U \), to signify that the subset inclusion \( V \incl U \) is a submolecule inclusion. 
If, furthermore, \( \dim V = \dim U \) and \( V \) is round, we say that \( V \) is a \emph{rewritable submolecule}. 
The name comes from the fact that, in this case, we can substitute (\emph{rewrite}) the image of \( V \) in \( U \) with any other round molecule \( W \) whose boundaries are isomorphic to those of \( V \), and obtain another molecule; see \cite{hadzihasanovic_higher-dimensional_2023}.

It is proven in \cite[Section 5.3]{amar_pasting} that isomorphism classes of molecules form a strict \( \omega \)\nbd category with \( - \cp{k} - \) as \( k \)-composition.
In particular, we record the following expressions for the boundaries of a pasting.
\begin{prop*}
    Let \( U, V \) be molecules such that \( U \cp i V \) is defined. Then
    \begin{equation*}
        \begin{cases}
            \bd^\a_k ( U \cp i V) = \bd^\a_k U = \bd^\a_k V \text{ for } k < i, \\
            \bd^-_i ( U \cp i V) = \bd^-_i U, \\
            \bd^+_i ( U \cp i V) = \bd^+_i V, \\
            \bd^\a_k ( U \cp i V) = \bd^\a _k U \cp i \bd^\a_k V \text{ for } k > i.
        \end{cases}
    \end{equation*} 
\end{prop*}

\subsection*{Acknowledgements}

We thank Philip Hackney for helpful comments on an earlier draft.
The first-named author is grateful to Marco Giustetto for useful discussions about the theory of presheaves over an Eilenberg--Zilber category.
The second-named author was supported by Estonian Research Council grant PSG764.

\section{The category of atoms and cartesian maps} \label{sec:atom}

\subsection{Definition and Eilenberg--Zilber property}

Recall from \cite[\S 6.2.1]{amar_pasting} that a map \( f \colon P \to Q \) of regular directed complexes is a closed order-preserving map of their underlying posets, which, for all \( x \in P \), \( n \in \mathbb N \), and \( \a \in \set{ -, + } \), satisfies
\begin{equation*}
    f(\bd^\a_n x) = \bd^\a_n f(x),
\end{equation*}
and, furthermore, the surjection \( \restr f {\bd^\a_n x} \colon \bd^\a_n x \to \bd^\a_n f(x) \) is \emph{final}, which in this context means that, for all \( y, y' \in \bound{n}{\a}x \), if \( f(y) = f(y') \), then there exists a zig-zag \( y \leq y_1 \geq \ldots \leq y_m \geq y' \) in \( \bound{n}{\a}x \) such that \( f(y) \leq f(y_i) \) for all \( y \in \set{1, \ldots, m} \).
An \emph{inclusion} is an injective map.
By \cite[Lemma 6.2.3]{amar_pasting}, maps are dimension-non-increasing.

\begin{dfn}[Cartesian map]
    Let \( f \colon P \to Q \) be an order-preserving map of posets.
    We say that \( f \) is \emph{cartesian} if it is a Grothendieck fibration between \( P \) and \( Q \) seen as posetal categories.
    Explicitly, \( f \) is cartesian if, for all \( x \in P \), there exists a \emph{cartesian lift} of each \( y \le f(x) \), that is, some \( y' \le x \) such that
    \begin{itemize}
        \item \( f(y') = y \), and
        \item for all \( z \le x \), if \( f(z) \le y \), then \( z \le y' \).
    \end{itemize}
    We say that a map of regular directed complexes is cartesian if its underlying map of posets is cartesian.
\end{dfn}

\noindent We let \( \rdcpx \) be the category of regular directed complexes and cartesian maps, and we let \( \atom \) (\emph{atom}) be a skeleton of its full subcategory on the atoms.

\begin{rmk} \label{rmk:cartesian_maps}
	The only difference between \( \rdcpx \) and \( \rdcpxmap \), as defined in \cite[Section 6.2]{amar_pasting} is that \( \rdcpx \) has strictly fewer surjections: all inclusions (and more in general, all \emph{local embeddings}, which are discrete Grothendieck fibrations) are already cartesian, but not all surjective maps are.
	For instance, consider the cubical coconnection map \( c \eqdef \cubeconn{}{-} \colon I \gray I \to I \) from \cite[\S 9.3.13]{amar_pasting}:
        \[
	\begin{tikzcd}
            {0^-} & {0^+} && {0^+} \\
            {0^-} & {0^-} && {0^-}
            \arrow["1", from=1-1, to=1-2]
            \arrow["{0^-}", from=2-1, to=1-1]
            \arrow["{0^-}"', from=2-1, to=2-2]
            \arrow["1"{description}, Rightarrow, from=2-2, to=1-1]
            \arrow[""{name=0, anchor=center, inner sep=0}, "1"', from=2-2, to=1-2]
            \arrow[""{name=1, anchor=center, inner sep=0}, "1"', from=2-4, to=1-4]
            \arrow["c", shorten <=13pt, shorten >=13pt, maps to, from=0, to=1]
        \end{tikzcd}
\]
    This is indeed a map of atoms, but \( c^{-1}(0^-) \) has no greatest element, so the map is not cartesian according to Lemma \ref{lem:fiber_has_top_element}.
\end{rmk}

\begin{rmk}\label{rmk:no_nontrivial_auto}
    By \cite[Proposition 6.3.13]{amar_pasting}, isomorphisms of regular directed complexes in \( \rdcpx \) coincide with their isomorphisms when seen as oriented graded posets.
    In particular, molecules still have no non-trivial automorphisms in \( \rdcpx \).
\end{rmk}

\begin{prop} \label{prop:point_is_terminal}
	The point \( \pt \) is a terminal object in \( \rdcpx \) and in \( \atom \).
\end{prop}
\begin{proof}
	By \cite[Proposition 6.2.9]{amar_pasting}, the unique function to \( \pt \) is a map of regular directed complexes, and it is evidently cartesian.
\end{proof}

\begin{prop}\label{prop:epi_mono_factorisation}
    Let \( f \colon P \to Q \) be a cartesian map of regular directed complexes. 
    Then \( f \) factors as 
    \begin{enumerate}
        \item a surjective cartesian map \( \hat f \colon P \surj f(P) \),
        \item followed by an inclusion \( i \colon f(P) \incl Q \).
    \end{enumerate}
    This factorisation is unique up to unique isomorphism.
\end{prop}
\begin{proof}
    From \cite[Proposition 6.2.26]{amar_pasting}, we have an essentially unique factorisation of \( f \) as a surjective map followed by an inclusion, and all inclusions are cartesian.
    Thus, it suffices to check that \( \hat f \) is cartesian, which follows immediately from the fact that \( f \) is cartesian.
\end{proof}

\noindent Proposition \ref{prop:epi_mono_factorisation} determines an orthogonal factorisation system on \( \rdcpx \) which restricts to an orthogonal factorisation system on \( \atom \).
We will call a surjective cartesian map of atoms a \emph{collapse}.

\begin{dfn}[Eilenberg--Zilber category]
    An \emph{Eilenberg--Zilber category} is a category \( \C \) together with two subcategories \( \C^+ \), \( \C^- \), and a map \( d \colon \Ob(\C) \to \mathbb N \) such that
    \begin{enumerate}[align=left]
        \item[{\crtcrossreflabel{(EZ0)}[enum:ez0]}] each isomorphism in \( \C \) belongs both to \( \C^+ \) and \( \C^- \), and \( d(a) = d(b) \) whenever \( a \) and \( b \) are isomorphic;
        \item[{\crtcrossreflabel{(EZ1)}[enum:ez1]}] if \( f\colon a \to b \) is a morphism in \( \C^+ \) (respectively, \( \C^- \)) that is not an identity, then \( d(a) < d(b) \) (respectively, \( d(a) > d(b) \));
        \item[{\crtcrossreflabel{(EZ2)}[enum:ez2]}] each morphism \( f \colon a \to b \) factors uniquely as \( ip \) with \( i \) in \( \C^+ \) and \( p \) in \( \C^- \);
        \item[{\crtcrossreflabel{(EZ3)}[enum:ez3]}] each morphism \( p \colon a \to b \) in \( \C^- \) has a section, and, given another morphism \( p' \colon a \to b \), the equation \( p = p' \) holds if and only if the sections of \( p \) coincide with the sections of \( p' \).
    \end{enumerate}
    An Eilenberg--Zilber category is \emph{regular} if, in addition, all morphisms in \( \C^+ \) are monomorphisms in \( \C \).
\end{dfn}

\begin{rmk}
	We follow \cite[Definition 1.3.1]{cisinski_higher_2019} for the notion of Eilenberg--Zilber category, which coincides with that of ``cat\'egorie squelettique normale'' in \cite{cisinski_prefaisceaux_2006}.
	By \cite[Proposition 8.2.2]{cisinski_prefaisceaux_2006}, a regular Eilenberg--Zilber category coincides with a ``cat\'egorie squelettique r\'eguliere'' in the sense of \cite[\S 8.2.3]{cisinski_prefaisceaux_2006}. 
	The latter is also called a ``regular skeletal Reedy category'' in \cite{campion2023cubical}.
\end{rmk}

\noindent The rest of this subsection will be devoted to showing that the classes of collapses and inclusions determine the structure of a regular Eilenberg--Zilber category on \( \atom \), with \( d \eqdef \dim \colon \Ob(\atom) \to \mathbb{N} \).

\begin{lem}\label{lem:fiber_has_top_element}
    Let \( p \colon U \surj V \) be a collapse of atoms.
    Then, for all \( y \in V \), the fibre \( \invrs{p}(y) \) has a greatest element.
\end{lem}
\begin{proof}
	Let \( \top_U \), \( \top_V \) be the greatest elements of \( U \) and \( V \), respectively.
	As \( p \) is surjective, we have \( y \le \top_V = p(\top_U) \), so there is a cartesian lift \( x \le \top_U \) with \( x \in \invrs{p}(y) \).
	Then, in particular, for all \( z \in U \) with \( p(z) = y \), we have \( z \le x \), meaning that \( x \) is the greatest element of \( \invrs{p}(y) \).
\end{proof}

\begin{dfn}[Minimal upper bound]
    Let \( P \) be a poset and \( x, y \in P \).
    A \emph{minimal upper bound} for \( x, y \) is an element \( z \in P \) such that
    \begin{enumerate}
        \item \( x, y \le z \), and
        \item for all \( z' \) such that \( x, y \le z' \le z \), we have \( z = z' \).
    \end{enumerate} 
\end{dfn}

\begin{lem}\label{lem:fiber_collapse_properties}
	Let \( p \colon U \surj V \) be a collapse of atoms, \( y \in V \).
    	Then
    \begin{enumerate}
	\item for all \( x, x' \in \invrs{p}(y) \), \( x \) and \( x' \) have a minimal upper bound in \( \invrs{p}(y) \),
	\item for all \( x, x' \in \invrs{p}(y) \), if \( z \) is a minimal upper bound of \( x, x' \) in \( U \), then \( z \in \invrs{p}(y) \),
	\item for all \( z \in \invrs{p}(y) \), either \( z \) is minimal in \( \invrs{p}(y) \) or it is a minimal upper bound of some \( x, x' \in \invrs{p}(y) \),
	\item for all \( z \in \invrs{p}(y) \), \( z \) is minimal if and only if \( \dim z = \dim y \).
    \end{enumerate}
\end{lem}

\begin{comm}
	Lemma \ref{lem:fiber_collapse_properties} should be interpreted as the statement that, for all \( y \in V \), the fibre \( \invrs{p}(y) \) is generated by its elements of dimension \( \dim{y} \) under minimal upper bounds in \( U \).
\end{comm}

\begin{proof}
	By Lemma \ref{lem:fiber_has_top_element}, \( \invrs{p}(y) \) has a greatest element, so any two elements have an upper bound, in particular (by finiteness of atoms) a minimal upper bound. 
	Next, suppose \( x, x' \in \invrs{p}(y) \) and \( z \) is a minimal upper bound of \( x, x' \) in \( U \).
	The cartesian lift of \( y \le p(z) \) is the greatest \( z' \le z \) such that \( p(z') = y \).
	Then \( x, x' \le z' \le z \), so by minimality \( z = z' \), hence \( p(z) = y \).
	Next, let \( z \in \invrs{p}(y) \), and suppose \( z \) is not minimal in the fibre.
    	Because maps are dimension-non-increasing, necessarily \( k \eqdef \dim z > \dim p(z) = \dim y \).
	By \cite[Lemma 6.2.4]{amar_pasting}, there exist \( x^+ \in \faces{}{+}z \) and \( x^- \in \faces{}{-}z \) such that \( p(x^+) = p(x^-) = y \).
	Since \( z \) covers both \( x^+ \) and \( x^- \) in \( U \), it is necessarily a minimal upper bound.
	The last point follows from \cite[Lemma 6.2.5]{amar_pasting}.
\end{proof}

\begin{lem} \label{lem:fibration_same_dim_is_dim_preserving}
    Let \( p \colon U \surj V \) be a collapse of atoms of the same dimension.
    Then \( p \) is dimension-preserving.
\end{lem}
\begin{proof}
    We proceed by induction on \( n \eqdef \dim U = \dim V \).
    If \( n = 0 \), then \( U = V = \pt \) and there is nothing to do. 
    Suppose \( n > 0 \).
    For each \( x \in \faces{}{} \top_U \), we have some \( y \in V \) such that \( p(x) \le y \in \faces{}{} \top_V = \faces{}{} p(\top_U) \). 
    Let \( x' \le \top_U \) be the cartesian lift of \( y \le p(\top_U) \); 
    by definition, \( x \le x' \), but \( x \in \faces{}{} \top_U \), so either \( x' = \top_U \), in which case \( \top_V = p(x') = y \in \faces{}{} \top_V \) is a contradiction, or \( x = x' \), hence \( p(x) = y \in \faces{}{} \top_V \). 
    This proves that the map
	    \( \restr p {\clset{x}} \colon \clset{x} \to p(\clset{x}) \)
    is a collapse of atoms of the same dimension, so by the inductive hypothesis, it is dimension-preserving.
    Since \( x \) was arbitrary in \( \faces{}{} \top_U \) and, by assumption, \( \dim p(\top_U) = \dim \top_U \), we conclude that \( p \) is dimension-preserving.
\end{proof}

\begin{cor}\label{cor:fibration_same_dim_is_iso}
    Let \( p \colon U \surj V \) be a collapse of atoms of the same dimension.
    Then \( p \) is an isomorphism.
\end{cor}
\begin{proof}
	Follows from Lemma \ref{lem:fibration_same_dim_is_dim_preserving} and 
	\cite[Proposition 6.2.18]{amar_pasting}.
\end{proof}

\begin{lem}\label{lem:collapse_thin_interval}
	Let \( p \colon P \to Q \) be a cartesian map of regular directed complexes, and let \( z < x \in P \) with \( \dim{x} = \dim{z} + 2 \).
	If \( p(x) \in \faces{}{} p(z) \), then there exists \( y \in \faces{}{} x \) such that \( p(y) = p(x) \).
\end{lem}
\begin{proof}
	By oriented thinness, \( \cofaces{}{} z \cap \faces{}{} x = \set{y_0, y_1} \) for precisely two elements \( y_0, y_1 \). 
    	Let \( z' \le x \) be a cartesian lift of \( p(z) < p(x) \), and suppose by way of contradiction that \( p(y_0), p(y_1) \neq p(x) \).
	Then necessarily \( p(y_0) = p(y_1) = p(z) \), hence \( y_0, y_1 \le z' \le x \).
	Since \( y_0, y_1 \in \faces{}{}x \), necessarily \( z' = x \), implying that \( p(z) = p(z') = p(x) \), a contradiction.
\end{proof}

\begin{lem}\label{lem:interval_contract_exists_face}
    Let \( p \colon P \to Q \) be a cartesian map of regular directed complexes, \( x \le y \) in \( P \), and suppose that \(\dim p(y) - \dim p(x) < \dim y - \dim x\).
    Then there exists \( y' \in \faces{}{} y \) such that \( x \le y' \) and \( p(y') = p(y) \).
\end{lem}
\begin{proof}
	Let \( n \eqdef \dim y - \dim x \).
	If \( n = 1 \), then \( x \in \faces{}{} y \) and \( p(x) = p(y) \), so \( y' \eqdef  x \).
	Suppose \( n > 0 \) and let \( x = y_0 \prec y_1 \ldots \prec y_n = y \) be a chain in the covering diagram of \( P \).
    Since \( \dim p(y) - \dim p(x) < \dim y - \dim x \), there exists a largest index \( i \) such that \( p(y_i) = p(y_{i - 1}) \), and we can choose a chain such that \( i \) is maximal.
    We claim that \( i = n \).
    Indeed, if \( i < n \), by oriented thinness, we can complete the chain \( y_{i-1} \prec y_{i} \prec y_{i+1} \) to the diamond
    \[
	    \begin{tikzcd}[sep=small]
            & {y_{i + 1}} \\
            {y_i} && {y'_i} \\
            & {y_{i - 1}}
            \arrow[no head, from=2-1, to=1-2]
            \arrow[no head, from=3-2, to=2-1]
            \arrow[no head, from=3-2, to=2-3]
            \arrow[no head, from=2-3, to=1-2]
        \end{tikzcd}
\]
	for exactly one \( y'_i \in P \).
    By assumption, \( p(y_{i - 1}) = p(y_i) \in \faces{}{} p(y_{i + 1}) \), so Lemma \ref{lem:collapse_thin_interval} applies, yielding \( p(y'_i) = p(y_{i + 1}) \).
    Therefore the chain
    \[
        x = y_0 \prec \ldots \prec y_{i - 1} \prec y'_i \prec y_{i + 1} \prec \ldots \prec y_n = y 
\]
    is such that \( p(y'_i) = p(y_{i + 1}) \), contradicting the maximality of \( i \). 
    We conclude that \( p(y) = p(y_{n-1}) \), and the statement follows with \( y' \eqdef y_{n-1} \).
\end{proof}

\begin{lem}\label{lem:minimal_element_less_than_minimal_element_top_fiber}
	Let \( p \colon U \surj V \) be a collapse of atoms, \( y \in V \), and let \( x \) be a minimal element of \( \invrs{p}(y) \).
	Then there exists a minimal element \( z \) of \( \invrs{p}(\top_V) \) such that \( x \le z \). 
\end{lem}
\begin{proof}
	Let \( k \eqdef \dim{V} - \dim{U} \); we will construct a chain \( x_k \prec \ldots \prec x_0 \) in \( \invrs{p}(\top_V) \) such that \( x \le x_i \) for all \( i \in \set{0, \ldots, k} \).
	We let \( x_0 \eqdef \top_U \).
	If \( k = 0 \), then we are done.
	Otherwise, suppose that we have constructed \( x_i \) for \( i < k \).
	By Lemma \ref{lem:fiber_collapse_properties}, we have \( \dim x = \dim y \), so
    	\[
        	\dim p(x_i) - \dim p(x) = \dim V - \dim y = \dim V - \dim x < \dim x_i - \dim x.
	\]
    	Thus Lemma \ref{lem:interval_contract_exists_face} applies, and we can find \( x_{i+1} \in \faces{}{} x_i \) such that \( p(x_{i+1}) = \top_V \) and \( x \le x_{i+1} \).
	Since \( \dim{x_k} = \dim{\top_V} \), by Lemma \ref{lem:fiber_collapse_properties}, \( x_k \) is minimal in the fibre, and the statement holds for \( z \eqdef x_k \).
\end{proof}

\begin{prop}\label{prop:atom_ez_category}
    Let
	\begin{enumerate}
	\item \( \atom^+ \) be the wide subcategory of \( \atom \) on the inclusions, 
	\item \( \atom^- \) be the wide subcategory of \( \atom \) on the collapses, and 
	\item \( d \eqdef \dim \colon \Ob(\atom) \to \mathbb N \).
	\end{enumerate}
	With this structure, \( \atom \) is a regular Eilenberg--Zilber category.
\end{prop}
\begin{proof}
	Any isomorphism is both a collapse and an inclusion, so \ref{enum:ez0} is true.
	If \( i \colon U \incl V \) is a non-trivial inclusion, then \( \top_V \) is not in the image of \( i \), hence \( \dim U < \dim V \).
	If \( p \colon U \surj V \) is a non-trivial collapse, then \( \dim U > \dim V \) by the contrapositive of Corollary \ref{cor:fibration_same_dim_is_iso}.
	This proves \ref{enum:ez1}. 
	Next, \ref{enum:ez2} follows from Proposition \ref{prop:epi_mono_factorisation} combined with the fact that \( \atom \) is skeletal and that atoms have no non-trivial automorphisms.

	To conclude, it suffices to prove \ref{enum:ez3}. 
	Let \( p \colon U \surj V \) be a collapse, and take any minimal element \( z \in \invrs{p}(\top_V) \).
	By Lemma \ref{lem:fiber_collapse_properties}, \( \dim z = \dim p(z) = \top_V \), so the map
	\( \restr{p}{\clset{z}} \colon \clset{z} \to V \)
    	is an isomorphism. 
	It follows that the inclusion \( \iota_z \colon \clset{z} \incl U \) is a section of \( p \). 
	Conversely, if \( \iota \colon V \incl U \) is a section of \( p \), then \( \iota(\top_V) \) is a minimal element of \( \invrs{p}(\top_V) \).
	This establishes that \( p \) has a section, and that its sections are in bijection with the minimal elements of \( \invrs{p}(\top_V) \).

	Let \( p' \colon U \to V \) be another collapse with the same sections as \( p \), let \( y \in V \), and let \( x \) be a minimal element of \( \invrs{p}(y) \).
	By Lemma \ref{lem:minimal_element_less_than_minimal_element_top_fiber}, \( x \in \clset{z} \) for some minimal element \( z \in \invrs{p}(\top_V) \).
	Let \( \iota \colon V \incl U \) be the section of \( p \) determined by \( z \); then \( \iota \) is also a section of \( p' \), so \( p(x) = p'(x) = y \).
	It follows that the minimal elements of \( p^{-1}(y) \) coincide with the minimal elements of \( p'^{-1}(y) \). 
	By Lemma \ref{lem:fiber_collapse_properties}, \( \invrs{p}(y) \) and \( \invrs{p'}(y) \) can be reconstructed as the closures of their minimal elements under minimal upper bounds in \( U \), so \( \invrs{p}(y) = \invrs{p'}(y) \).
	Since \( y \) was arbitrary, we conclude that \( p = p' \).

	Finally, the functor sending a map of atoms to its underlying function of sets is evidently faithful, which implies that inclusions are monomorphisms in \( \atom \).
	This completes the proof.
\end{proof}

%%%%%%%%%%%%%%

\subsection{Gray products and joins}

In this section, we recall the definitions of the Gray product and join of regular directed complexes; see \cite[Chapter 7]{amar_pasting} for a more thorough treatment.

\begin{dfn}[Gray product]
    Let \( P, Q \) be oriented graded posets. The \emph{Gray product} of \( P \) and \( Q \) is the oriented graded poset \( P \gray Q \) whose
    \begin{itemize}
        \item underlying graded poset is the cartesian product \( P \times Q \),
        \item orientation is defined, for all \( (x, y) \in P \times Q \) and \( \a \in \set{ -, + } \), by the equation \( \Delta^\a(x, y) = \Delta^\a x \times \set{ y } + \set{ x } \times \Delta^{(-)^{\dim x} \a} y \).
    \end{itemize}
\end{dfn}

\begin{prop}\label{prop:gray_monoidal_rdcpx}
    The Gray product determines a monoidal structure on \( \rdcpx \), whose unit is the point \( \pt \).
\end{prop}
\begin{proof}
    By \cite[Proposition 7.2.19]{amar_pasting}, it suffices to show that, if \( f, g \) are cartesian, then \( f \gray g \) is cartesian.
    This follows from the general fact that Grothendieck fibrations are closed under cartesian products.
\end{proof}

\begin{lem}\label{lem:gray_preserves_epi_mono}
	The Gray product of two atoms is an atom, the Gray product of two inclusions is an inclusion, and the Gray product of two surjective cartesian maps is surjective.
\end{lem}
\begin{proof}
	The product of two posets with greatest element has a greatest element, the product of two injective functions is injective, and the product of two surjective functions is surjective.
\end{proof}

\begin{cor}\label{cor:monoidal_on_atom}
	The monoidal structure \( (\rdcpx, \gray, \pt) \) restricts to monoidal structures on \( \atom \), \(\atom^+ \), and \( \atom^- \).
\end{cor}

\begin{dfn}[Arrow]
	The \emph{arrow} is the atom \( \thearrow{} \eqdef 1 \celto 1 \). 
	We denote by \( 0^- < 1 > 0^+ \) the three elements of its underlying poset, with \( \faces{}{\a} 1 = \set{0^\a} \).
\end{dfn}

\begin{dfn}[Cylinder]
    Let \( U \) be a molecule. 
    The \emph{cylinder} on \( U \) is the molecule \( \thearrow{} \gray U \). It is equipped with a projection \( \sigma \colon \thearrow{} \gray U \surj U \) induced by the unique map \( \varepsilon \colon \thearrow{} \surj \pt \) to the terminal object, and two sections
    \( \iota^+, \iota^- \colon U \incl \thearrow{} \gray U \)
    induced by the two inclusions \( 0^\a\colon \pt \incl \thearrow{} \) for \( \a \in \set{ +, - } \).
\end{dfn}

\noindent If \( P \) is a poset, let \( \augm P \) be the poset obtained by freely adding a least element \( \bot \).
Conversely, if \( P \) is a poset with least element \( \bot \), let \( \dimin P \) be obtained by restricting to \( P \setminus \set{\bot} \).
As detailed in \cite[Section 1.3]{amar_pasting}, these two operations extend to functors that determine an equivalence of categories between
\begin{itemize}
	\item the category of posets and closed order-preserving maps, and
	\item the category of posets with a least element and closed order-preserving maps that reflect the least element.
\end{itemize}

\begin{lem}\label{lem:fibration_closed_augm_dimin}
    Let \( f \colon P \to Q \) be a closed cartesian map of posets. 
    Then
    \begin{enumerate}
	    \item \( \augm f \colon \augm P \to \augm Q \) is cartesian,
		\item if \( P, Q \) have a least element reflected by \( f \), then \( \dimin f \colon \dimin P \to \dimin Q \) is cartesian.
    \end{enumerate}
\end{lem}
\begin{proof}
	Let \( x \in \augm{P} \) and let \( y \le \augm{f}(x) \) in \( \augm Q \). 
	If \( y = \bot \), then \( \bot \le x \) is a cartesian lift of \( y \le \augm f(x) \) to \( \augm P \). 	  Else, \( y \) is not \( \bot \), hence \( \augm{f}(x) \) and \( x \) are not \( \bot \) either, and \( \augm{f}(x) = f(x) \).
        Since \( f \) is cartesian, there exists a cartesian lift \( z \le x \) of \( y \le f(x) \) to \( P \). 
	We claim that this lift is again cartesian for \( \augm f \).
	Suppose \( z' \le x \) such that \( \augm f(z') \le y \).
	If \( z' = \bot \), then \( z' \le z \) by construction, otherwise \( \augm f(z') = f(z') \) and \( z' \le z \) by cartesianness of \( f \).

	Next, suppose \( P, Q \) have a least element, let \( x \in \dimin P \), and let \( y \le \dimin{f}(x) \) in \( \dimin Q \).
	Since \( f \) is cartesian, there is a cartesian lift \( z \le x \) of \( y \le f(x) \) to \( P \). 
	Suppose by way of contradiction that \( z = \bot \).
	Then \( f(z) = \bot \), contradicting \( y \in \dimin Q \). 
	Thus \( z \le x \) is in \( \dimin P \), and the cartesian property of this lift for \( \dimin f \) directly follows from the one for \( f \).
\end{proof}

\noindent As shown in \cite[Section 2.3]{amar_pasting}, the pair of functors \( \augm{(-)}, \dimin{(-)} \) can be lifted to an equivalence between the category of oriented graded posets and its full subcategory on oriented graded posets with a \emph{positive} least element, that is, a least element \( \bot \) with \( \cofaces{}{}\bot \equiv \cofaces{}{+}\bot \).
The latter are closed under Gray products, and transporting this monoidal structure along the equivalence determines a separate monoidal structure on \( \ogpos \).

\begin{dfn}[Join]
	Let \( P, Q \) be oriented graded posets.
	The \emph{join} \( P \join Q \) of \( P \) and \( Q \) is the oriented graded poset \( \dimin{(\augm P \gray \augm Q)} \).
\end{dfn}

\begin{prop}\label{prop:join_monoidal}
    The join determines a monoidal structure on \( \rdcpx \), whose unit is the empty regular directed complex \( \varnothing \).
\end{prop}
\begin{proof}
    Since the join defines a monoidal structure on \( \rdcpxmap \) by \cite[Proposition 7.4.22]{amar_pasting}, it suffices to show that the join of two cartesian maps is cartesian.
    Since the underlying map of \( f \join g \) is \( \dimin {(\augm f \times \augm g)} \) and cartesian maps are closed under products, this follows from Lemma \ref{lem:fibration_closed_augm_dimin}. 
\end{proof}

\noindent
It follows that the embedding of the simplex category \( \simplexcat \) into \( \rdcpxmap \) defined in \cite[Proposition 9.2.14]{amar_pasting} factors through \( \rdcpx \).
In particular, we may identify \( \simplexcat \) with a full subcategory of \( \atom \).

\section{Diagrammatic sets} \label{sec:diagrammatic}

\begin{dfn}[Diagrammatic set]
    A \emph{diagrammatic set} is a presheaf on \( \atom \).
    Diagrammatic sets and natural transformations form the category \( \atom\Set \).
\end{dfn}

\noindent Given an atom \( U \) and a diagrammatic set \( X \), a \emph{cell of shape \( U \) in \( X \)} is an element \( x \in X(U) \). 
We will identify atoms with their Yoneda embedding, and make no distinction between a cell \( x \in X(U) \) and its classifying morphism \( x \colon U \to X \).

\begin{dfn}[Non-degenerate cell]
	Let \( X \) be a diagrammatic set and \( x \colon U \to X \) a cell.
	We say that \( x \) is \emph{non-degenerate} if, for all collapses \( p\colon U \surj V \) and cells \( y\colon V \to X \), if \( x = yp \) then \( p = \idd{U} \) and \( y = x \).
	We say that \( x \) is \emph{degenerate} otherwise.
	We write \( \Nd X \) for the set of all non-degenerate cells of \( X \).
\end{dfn}

\begin{lem}\label{lem:ez_lemma}
    Let \( X \) be a diagrammatic set, and let \( x \colon U \to X \) be a cell. 
    Then there exists a unique pair \( (p \colon U \surj V, y \colon V \to X) \) such that
    \( p \) is a collapse, \( y \) is non-degenerate, and \( x = yp \).
\end{lem}
\begin{proof}
    This is a general result about Eilenberg--Zilber categories, see for instance \cite[Lemma 1.3.6]{cisinski_higher_2019}.
\end{proof}
\noindent We will call \( x = yp \) the \emph{Eilenberg--Zilber decomposition} of the cell \( x \).

Let \( i \colon \atom \incl \rdcpx \) be the inclusion of \( \atom \) as a full subcategory.
Each regular directed complex \( P \) can be seen as a presheaf on \( \atom\Set \) via the functor 
\begin{equation*}
    i^* \colon \rdcpx \to \atom\Set, \quad \quad P \mapsto \rdcpx(i(-), P).
\end{equation*}
We will show that this functor embeds \( \rdcpx \) as a full subcategory of \( \atom\Set \), that is, the functor \( i \) is dense.
Furthermore, by a variant of \cite[Lemma 6.2.25]{amar_pasting}, a regular directed complex \( P \) is the colimit of the diagram of inclusions of its atoms in the category \( \rdcpx \); we will show that this colimit is preserved by \( i^* \). 
Finally, we will characterise the essential image of \( i^* \) as the ``regular presheaves'' of \( \atom\Set \).
In the sequel, this will allow us to make no distinction between a regular directed complex \( P \) and its image \( i^*P \).

\begin{lem}\label{lem:nondeg_in_rdcpx}
	Let \( P \) be a regular directed complex and \( f\colon U \to P \) a cell in \( i^* P \).
	Then \( f \) is non-degenerate as a cell if and only if it is an inclusion.
\end{lem}
\begin{proof}
	Follows from Proposition \ref{prop:epi_mono_factorisation} combined with the non-existence of non-trivial isomorphisms in \( \atom \).
\end{proof}

\begin{lem}\label{lem:rdcpx_to_dset_ff}
   The functor \( i^* \colon \rdcpx \to \atom\Set \) is full and faithful.
\end{lem}
\begin{proof}
	Let \( P \) be a regular directed complex.
	For each \( x \in P \), we denote by \( \iota_x \colon U_x \incl P \) the unique inclusion with image \( \clset{x} \) for some object \( U_x \) in \( \atom \).
	Let \( f, g \colon P \to Q \) be two maps in \( \rdcpx \), and suppose that \( i^*f = i^*g \).
	Then, for all \( x \in P \), we have \( i^*f(\iota_x) = i^*g(\iota_x) \), which implies that \( \restr f {\clset{x}} = \restr g {\clset{x}} \), hence \( f = g \).
	This proves faithfulness.

    	Next, let \( f \colon i^*P \to i^*Q \) be a morphism in \( \atom\Set \). 
	We define \( \hat{f} \colon P \to Q \) by \( x \in P \mapsto f(\iota_x)(\top_{U_x}) \in Q \). 
	Let \( x \in P \); we prove that \( f(\iota_x) = \hat{f} \iota_x \).
	Let \( y \in U_x \) and \( y' \eqdef \iota_x(y) \in P \).
	By naturality of \( f \), we have 
	\[ f(\iota_x)(y) = f(\iota_{y'})(\top_{U_{y'}}) = \hat{f}(y') = \hat{f} \iota_x (y). \]
	Since \( f(\iota_x) \) and \( \iota_x \) are cartesian maps of regular directed complexes, it follows that \( \restr{\hat{f}}{\clset{x}} \) is a cartesian map, and since \( x \) is arbitrary, \( \hat{f} \) is well-defined as a cartesian map \( P \to Q \).

	Let \( x \colon U \to P \) be any cell in \( i^*P \).
	Factorising \( x \) as \( yp \) for a unique pair of a collapse \( p\colon U \surj V \) and a non-degenerate cell \( y\colon V \to P \) in \( i^*P \), by Lemma \ref{lem:nondeg_in_rdcpx} \( y \) is an inclusion in \( \rdcpx \), so it is equal to \( \iota_{y'} \) for a unique \( y' \in P \).
	Then
    	\begin{equation*}
		f(x) = f(\iota_{y'}) p = (\hat{f} \iota_{y'}) p = i^*\hat{f}(\iota_{y'}) p = i^*\hat{f}(x),
    	\end{equation*}
	proving that \( f = i^* \hat{f} \).
	This proves fullness.
\end{proof}

\begin{prop}\label{prop:colim_inclusions_preserved_by_i^*}
	Let \( \fun{F} \colon \C \to \rdcpx \) be a diagram of inclusions of regular directed complexes, and let \( \gamma \) be a colimit cone under \( \fun{F} \) in \( \rdcpx \) whose components are all inclusions. 
	Then \( i^* \gamma \) is a colimit cone under \( i^* \fun{F} \) in \( \atom\Set \).
\end{prop}
\begin{proof}
	Let \( P \) be the colimit of \( \fun{F} \) in \( \rdcpx \), and let 
    \begin{equation*}
	    \phi \colon \colim_c \rdcpx(i(-), \fun{F}c) \to \rdcpx(i(-), P)
    \end{equation*}
    be the universal morphism determined by the cone \( i^*\gamma \). 
    We will show that \( \phi \) is an isomorphism. 
    Let \( U \) be an atom, let 
    \begin{equation*}
        f, g \in \colim_c \rdcpx(i(U), \fun{F}c)
    \end{equation*}
    be represented respectively by \( f \colon U \to \fun{F}c \) and \( g \colon U \to \fun{F}c' \), and suppose \( \phi(f) = \phi(g) \).
    Then \( \gamma_c f(\top_U) = \gamma_{c'} g(\top_U) \) in \( P \), so by the construction of \( P \) as a colimit, there exists a zig-zag of morphisms in \( \C \)
    \begin{equation*}
	    c \stackrel{u_0}{\longleftarrow} c_1 \stackrel{u_1}{\longrightarrow} \ldots \stackrel{u_{n - 2}}{\longleftarrow} c_{n - 1} \stackrel{u_{n-1}}{\longrightarrow} c'
    \end{equation*}
    and a sequence \( (x_i \in \fun{F}c_i)_{0 \le i \le n} \) such that \( x_0 = f(\top_U) \), \( x_n = g(\top_U) \), and for all \( k \in \set{ 1, \ldots, \frac n 2} \),
    \begin{equation*}
        x_{2k - 2} = \fun{F}u_{2k - 2}(x_{2k - 1}), \quad \quad x_{2k} = \fun{F}u_{2k - 1}(x_{2k - 1});
    \end{equation*}
    note that the equations are only required to hold for the images via \( \gamma_{c_i} \), but we can lift those to \( \fun{F}c_i \) because the \( \gamma_{c_i} \) are injective. 
    Because all the \( \fun{F}u_i \) are injective, they induce isomorphisms \( \theta_i \colon \clset{x_i} \cong \clset{x_{i + 1}} \). 
    Let \( \iota_c\hat f \) and \( \iota_{c'} \hat g \) be factorisations of \( f \) and \( g \) as surjective maps followed by inclusions, and consider the diagram
    \[
	    \begin{tikzcd}
            && U \\
		{\clset{f(\top_U)}} & {\clset{x_1 }} & \ldots & {\clset{x_{n-1}}} & {\clset{g(\top_U)}} \\
            \fun{F}c & {\fun{F}c_1} & \ldots & {\fun{F}c_{n - 1}} & {\fun{F}c'} \\
            && P
            \arrow["{{\hat f}}"', two heads, from=1-3, to=2-1]
            \arrow["{h_1}", dashed, two heads, from=1-3, to=2-2]
            \arrow["{h_{n-1}}"', dashed, two heads, from=1-3, to=2-4]
            \arrow["{{\hat g}}", two heads, from=1-3, to=2-5]
            \arrow["\cong"{description}, draw=none, from=2-1, to=2-2]
            \arrow["{{\iota_c}}", hook', from=2-1, to=3-1]
            \arrow["\cong"{description}, draw=none, from=2-2, to=2-3]
            \arrow["{\iota_1}", hook', from=2-2, to=3-2]
            \arrow["\cong"{description}, draw=none, from=2-3, to=2-4]
            \arrow["\cong"{description}, draw=none, from=2-4, to=2-5]
            \arrow["{\iota_{n-1}}", hook', from=2-4, to=3-4]
            \arrow["{\iota_{c'}}", hook', from=2-5, to=3-5]
            \arrow["{\gamma_c}"', hook, from=3-1, to=4-3]
            \arrow["{\fun{F}u_0}"', hook', from=3-2, to=3-1]
            \arrow["{\fun{F}u_1}", hook, from=3-2, to=3-3]
	    \arrow["{\gamma_{c_{1}}}", hook, from=3-2, to=4-3]
            \arrow["{\fun{F}u_{n -1}}"', hook', from=3-4, to=3-3]
            \arrow["{\fun{F}u_n}", hook, from=3-4, to=3-5]
            \arrow["{\gamma_{c_{n-1}}}"', hook', from=3-4, to=4-3]
            \arrow["{\gamma_{c'}}", hook', from=3-5, to=4-3]
        \end{tikzcd}
\]
    where \( h_i \) is defined inductively by \( h_1 \eqdef \theta_0 \hat f \) and \( h_i \eqdef \theta_{i - 1} h_{i - 1} \) for \( i > 1 \). 
    The only equation that does not hold by construction or by assumption is \( \theta_{n-1} h_{n-1} = \hat g \), but this follows from
    \[
	    \gamma_{c'} \iota_{c'} \hat g = \gamma_{c} \iota_c \hat{f} = \gamma_{c'} \iota_{c'} \theta_{n-1} h_{n-1}
    \]
    and the fact that \( \gamma_{c'} \iota_{c'} \) is injective.
    This exhibits a zig-zag of morphisms between \( f \) and \( g \) under \( U \), proving that \( f = g \) in \( \colim_c \rdcpx(i(-), \fun{F}c) \). 
	This proves that \( \phi \) is injective on cells.

	Let \( f \colon U \to P \) in \( \rdcpx(i(U), P) \).
	Then \( f(\top_U) = \gamma_c(x) \) for some \( x \in \fun{F}c \), and since \( \gamma_c \) is an inclusion, \( f \) factors as \( \gamma_c \hat{f} \) for some \( \hat{f}\colon U \surj \fun{F}c \).
	It follows that \( f = \phi(\hat{f}) \), so \( \phi \) is also surjective on cells.
\end{proof}

\begin{cor}\label{cor:rdcpx_colim_its_atom}
    Let \( P \) be a regular directed complex.
    Then \( i^*P \) is the colimit of the diagram of inclusions
    \begin{equation*}
	    P \to \atom\Set,  \quad \quad (x \le y) \mapsto (i^*\clset{x} \incl i^*\clset{y}).
    \end{equation*}
\end{cor}
\begin{proof}
	Follows from \cite[Proposition 2.2.22]{amar_pasting} and Proposition \ref{prop:colim_inclusions_preserved_by_i^*}.
\end{proof}

\begin{cor}\label{cor:gray_product_preserved}
	The functor \( i^*\colon (\rdcpx, \gray, \pt) \incl (\atom\Set, \gray, \pt) \) is strong monoidal.
\end{cor}
\begin{proof}
	Follows from a variant of \cite[Lemma 7.2.8]{amar_pasting}, combined with the fact that each regular directed complex is the colimit of the diagram of inclusions of its atoms, and that the Yoneda embedding is strong monoidal.
\end{proof}

\begin{rmk}\label{rmk:notation_boundary}
    If \( \C \) is an Eilenberg--Zilber category and \( c \) is an object of \( \C \), one usually defines \( \bd c \) to be the subpresheaf of \( c \) generated by the morphisms \( c' \to c \) in \( \C^+ \) with \( d(c') < d(c) \). 
	If \( U \) is an atom, it follows from Corollary \ref{cor:rdcpx_colim_its_atom} that the subpresheaf \( \bd U \) of its Yoneda embedding is isomorphic to the image \( i^*\bd U \) of the regular directed complex \( \bd U \).
    
	In particular, by \cite[Theorem 1.3.8]{cisinski_higher_2019}, the inclusions \( \bd_U \colon \bd U \incl U \) for \( U \) an atom form a \emph{cellular model} \cite[\S 2.4.4]{cisinski_higher_2019} of \( \atom\Set \), in the sense that each monomorphism \( m \colon X \incl Y \) in \( \atom\Set \) is built out of a transfinite composition of pushouts of the \( \bd_U \).
\end{rmk}

\noindent We conclude this section by characterising the image of \( i^*\colon \rdcpx \to \atom\Set \).

\begin{dfn}[Regular diagrammatic set]
    A diagrammatic set \( X \) is \emph{regular} if all non-degenerate cells \( x \colon U \to X \) are monomorphisms.  
\end{dfn}

\noindent 
By Lemma \ref{lem:nondeg_in_rdcpx}, all regular directed complexes are regular as diagrammatic sets.
The rest of this section is devoted to showing that the converse holds up to isomorphism.

\begin{lem}\label{lem:subregular_is_non_degenerate}
    Let \( X \) be a regular diagrammatic set, let \( x \colon U \to X \) be a non-degenerate cell, and \( \iota\colon V \incl U \) an inclusion of atoms.
    Then \( x \iota \) is non-degenerate.
\end{lem}
\begin{proof}
    Let \( y p = x \iota \) be the Eilenberg--Zilber decomposition of \( x \iota \).
    Since both \( x \) and \( \iota \) are monomorphisms, \( p \) is also a monomorphism, hence the identity.
\end{proof}

\begin{prop}\label{prop:rdcpx_are_regular}
    A diagrammatic set \( X \) is regular if and only if it is isomorphic to \( i^* P \) for a regular directed complex \( P \).
\end{prop}
\begin{proof}
	One direction follows from Lemma \ref{lem:nondeg_in_rdcpx}.
    Conversely, suppose \( X \) is regular. 
    Then, as detailed in \cite[\S 8.2.21]{cisinski_prefaisceaux_2006}, the representable subobjects of \( X \) form a poset \( \xi X \).
    We claim that \( \xi X \) is in fact the underlying poset of a regular directed complex.
    If \( x \colon U \incl X \) is an element of \( \xi X \), then \( \clset{x} \) in \( \xi X \) is by definition the poset of representable subobjects of \( U \), which is isomorphic to the underlying poset of \( U \). 
    Thus we can endow \( \clset{x} \) with the orientation of \( U \), which makes it an atom, hence makes \( \xi X \) a regular directed complex.
    Let \( \fun{F} \colon \xi X \to \atom\Set \) be the functor sending a representable subobject \( x \colon U \incl X \) to \( U \).
    By Corollary \ref{cor:rdcpx_colim_its_atom}, \( \colim \fun{F} \cong i^* (\xi X) \), and by \cite[Lemme 8.2.22]{cisinski_prefaisceaux_2006}, \( X \cong \colim \fun{F} \), which completes the proof.
\end{proof}

\section{Homotopy theory of diagrammatic sets} \label{sec:homotopy}

\subsection{Cylinder object}

The goal of this subsection is to prove that the functor \( \thearrow{} \gray -\colon\atom\Set \to \atom\Set \) is an exact cylinder in the sense of \cite[\S 2.4.8]{cisinski_higher_2019}.
As a corollary, we will deduce that the category \( \atom \) is a test category, and we will further use the cylinder to characterise the model structure we obtain.

\begin{dfn}[Exact cylinder]
	Let \( \widehat\C \) be a category of presheaves on a small category \( \C \).
	An \emph{exact cylinder on \( \widehat\C \)} is an endofunctor \( \fun{I} \) on \( \widehat\C \) together with natural transformations \( (\iota^-, \iota^+)\colon \bigid{\widehat\C} \amalg \bigid{\widehat\C} \to \fun{I} \) and \( \sigma\colon \fun{I} \to \bigid{\widehat\C} \) such that
    \begin{enumerate}[align=left]
	   \item[{\crtcrossreflabel{(DH0)}[enum:dh0]}] each component of \( (\iota^-, \iota^+) \) is a monomorphism, and each component of the composite \( \sigma (\iota^-, \iota^+) \) is a codiagonal morphism;
	   \item[{\crtcrossreflabel{(DH1)}[enum:dh1]}] the functor \( \fun{I} \) preserves small colimits and monomorphisms.
    \end{enumerate}
\end{dfn}

\begin{rmk} \label{rmk:exact_cylinder_pullback}
	An exact cylinder is called a ``donn\'ee homotopique \'el\'ementaire'' in \cite[D\'efinition 1.3.6]{cisinski_prefaisceaux_2006}.
	Cisinski's definition includes the extra condition (DH2) requiring that, for all monomorphisms \( m\colon X \to Y \) and \( \a \in \set{+, -} \),
	\[\begin{tikzcd}
	X & Y \\
	{\fun{I}X} & {\fun{I}Y}
	\arrow["m", from=1-1, to=1-2]
	\arrow["{\iota^\a}", from=1-1, to=2-1]
	\arrow["{\iota^\a}", from=1-2, to=2-2]
	\arrow["{\fun{I}m}", from=2-1, to=2-2]
\end{tikzcd}\]
	be a pullback square, but as observed in \cite[Lemma 3.10]{olschok2009left} this is implied by the others.
\end{rmk}

\noindent Recall that, given a small monoidal category, we can extend its monoidal structure via Day convolution \cite{day1970closed} to its category of presheaves.
This monoidal structure is biclosed and uniquely characterised, among monoidal biclosed structures, by the fact that the Yoneda embedding is strong monoidal. 
Applied to \( (\atom, \gray, \pt) \), this gives \( (\atom\Set, \gray, \pt) \) a structure of biclosed monoidal category. 
Explicitly, given diagrammatic sets \( X, Y \), their Gray product \( X \gray Y \) is defined by
\[
    X \otimes Y\colon U \mapsto \int^{V, W \in \smallatom} X(V) \times Y(W) \times \atom(U, V \gray W),
\]
that is, cells \( x \colon U \to X \gray Y \) are represented by triples
\[    
	(y \colon V \to X, z \colon W \to Y, f \colon U \to V \gray W),
\]
quotiented by the equivalence relation generated by
\( (y, z, f) \sim (y', z', f')
\)
if and only if there exist cartesian maps of atoms \( g, h \) such that \(y = y'g \), \(z = z'h\), and \( (g \gray h)f = f' \).

\begin{lem}\label{lem:representative_nd_cell_day_convolution}
    Let \( X, Y \) be diagrammatic sets and let \( z \colon U \to X \gray Y \) be a non-degenerate cell.
    Then there exist non-degenerate cells \( x \colon V \to X \) and \( y \colon W \to Y \) such that \( U = V \gray W \) and \( x \) is represented by \( (x, y, \idd U) \).
\end{lem}
\begin{proof}
    Let \( (x' \colon V' \to X, y' \colon W' \to X, f \colon U \to V' \gray W') \) be a representative of \( z \). 
    Factorising \( f \) as a collapse \( p \) followed by an inclusion \( m \), since every atom in \( V' \gray W' \) is of the form \( \clset{(v, w)} \cong \clset{v} \gray \clset{w} \), the inclusion \(m\) must be of the form \(i \gray j\colon V \gray W \incl V' \gray W'\).
    Then the cell \( z' \) represented by
    \begin{equation*}
        (x' \colon V' \to X, y' \colon W' \to Y, i \gray j \colon V \gray W \to V' \gray W'),
    \end{equation*}
    is such that \( z = z' p \), and since \( z \) is non-degenerate \( p \) must be the identity. 
    Letting \( x \eqdef x'i \) and \( y \eqdef y'j \), we see that \( (x, y, \idd U) \) also represents \( z \).

    Suppose that \( x = \bar{x}q \) for some cell \( \bar{x}\colon \bar{V} \to X \) and collapse \( q \).
    Then \( z \) factors as \( (\bar{x}, y, \idd{\bar{V} \gray W})(q \gray \idd{W}) \), and by Lemma \ref{lem:gray_preserves_epi_mono} \( q \gray \idd{W} \) is a collapse.
    Since \( z \) is non-degenerate, \( q \gray \idd{W} \) must be the identity, hence \( q \) is an identity, which proves that \( x \) is non-degenerate.
    Dually, \( y \) is non-degenerate.
\end{proof}

\begin{rmk}\label{rmk:form_cell_day_convolution}
    Thus, any \( z \colon U \to X \gray Y \) can be written as 
    \begin{equation*}
        \begin{tikzcd}
            U & {V\gray W} & {X \gray Y}
            \arrow["{x \gray y}", from=1-2, to=1-3]
            \arrow["p", two heads, from=1-1, to=1-2]
        \end{tikzcd}
    \end{equation*}
    where \( x \colon V \to X \) and \( y \colon W \to Y \) are non-degenerate. 
    For any other decomposition \( z = (x' \gray y') p' \), we have \( V \gray W = V' \gray W' \) and \( p = p' \), but we cannot assume \( V = V' \) or \( W = W' \), unless \( X \) and \( Y \) are regular.
\end{rmk}

\begin{lem}\label{lem:day_gray_product_preserves_mono}
    Let \( m \colon X \incl Y \) and \( m' \colon X' \incl Y' \) be monomorphisms in \( \atom\Set \). 
    Then \( m \gray m' \colon X \gray X' \to Y \gray Y' \) is a monomorphism.
\end{lem}
\begin{proof}
    Since \( - \gray - \) preserves colimits in both variables, building on Remark \ref{rmk:notation_boundary}, we can construct \( m \gray m' \) as a transfinite composition of coproducts of pushouts of \( \bd_{U} \gray \idd{V} \) and \( \idd{U} \gray \bd_V \).
    These are monomorphisms in \( \rdcpx \) by Lemma \ref{lem:gray_preserves_epi_mono}, and the full and faithful functor \( i^*\colon \rdcpx \incl \atom\Set \) preserves monomorphisms.
    Since pushouts and transfinite compositions of monomorphisms are again monomorphisms in any presheaf category, we conclude.
\end{proof}

\begin{prop}\label{prop:exact_cylinder_object}
	The functor \( \thearrow{} \gray - \) together with the natural transformations \( (\iota^-, \iota^+) \) and \( \sigma \) induced by the maps
	\(
		(0^-, 0^+)\colon \pt \amalg \pt \incl \thearrow{} \) and 
	\( \varepsilon\colon \thearrow{} \surj \pt
	\), respectively,
	is an exact cylinder on \( \atom\Set \).
\end{prop}
\begin{proof}
	By Lemma \ref{lem:day_gray_product_preserves_mono}, \( \thearrow{} \gray - \) preserves monomorphisms, so in particular the components of \( (\iota^-, \iota^+) \) are monomorphisms.
	Since it is a left adjoint, it also preserves small colimits.
	Finally, the composite \( \varepsilon(0^-, 0^+) \) is the codiagonal on \( \pt \), which implies that each component of \( \sigma(\iota^-, \iota^+) \) is a codiagonal.
\end{proof}

\subsection{Model category structure}

\begin{dfn}[\( \infty \)\nbd equivalence]
   A morphism \( f \colon X \to Y \) of diagrammatic sets is an \emph{\( \infty \)\nbd equivalence} if the induced functor \( \slice{\atom}{f} \colon \slice{\atom}{X} \to \slice{\atom}{Y} \) is a weak equivalence in the Thomason model structure on \( \Cat \) \cite{thomason_model}.
\end{dfn}

\begin{lem}\label{lem:representable_map_are_infty_equivalences}
    Let \( f \colon U \to V \) be a cartesian map of atoms.
    Then \( f \) is an \( \infty \)\nbd equivalence.
\end{lem}
\begin{proof}
    The class of weak equivalences in the Thomason model structure on \( \Cat \) is a basic localiser. 
    Hence, by definition, for any category \( \C \) with a terminal object, the unique functor from \( \C \) to the terminal category \( \ptcat \) is a weak equivalence. 
    Since \( U, V \) are representable, the slice categories \( \slice{\atom}{U} \) and \( \slice{\atom}{V} \) have terminal objects \( \idd U \) and \( \idd V \), respectively. 
    Furthermore, the diagram of functors 
    \[    
    \begin{tikzcd}
		{\slice{\atom}{U}} && {\slice{\atom}{V}} \\
            & \ptcat
	    \arrow["{\slice{\atom}{f}}", from=1-1, to=1-3]
            \arrow[from=1-1, to=2-2]
            \arrow[from=1-3, to=2-2]
        \end{tikzcd}
\]
    strictly commutes, so we can conclude by 2-out-of-3 for weak equivalences.
\end{proof}

\begin{prop}\label{prop:atom_is_test}
    The category \( \atom \) is a test category.
\end{prop}
\begin{proof}
    As \( \pt \) is the terminal object, \( \atom \) is aspherical. 
    Therefore, it suffices to prove that \( \atom \) is a local test category.
    This follows from \cite[Corollaire 8.2.16]{cisinski_prefaisceaux_2006} using Proposition \ref{prop:atom_ez_category}, Proposition \ref{prop:exact_cylinder_object}, and the fact that \( \sigma \colon \thearrow{} \gray U \to U \) is an \( \infty \)\nbd equivalence by Lemma \ref{lem:representable_map_are_infty_equivalences}.
\end{proof}

\noindent Therefore, by the general theory of test categories, we obtain the following.

\begin{thm}\label{thm:cisinski_model_structure}
There is a cofibrantly generated model structure on \( \atom\Set \) whose cofibrations are the monomorphisms, and whose weak equivalences are the \( \infty \)\nbd equivalences.
\end{thm}

\noindent We will call this the \emph{Cisinski model structure} on \( \atom\Set \).
This model structure is Quillen-equivalent to the classical model structure on \( \sSet \) via the left Kan extension of the functor sending an atom \( U \) to the simplicial nerve of its category of elements.
In fact, we can obtain a simpler description of this Quillen equivalence as follows.

\begin{dfn}[Subdivision]\label{dfn:subdivision}
	The \emph{subdivision} functor \( \Sd_\smallatom \colon \atom\Set \to \sset \) is the left Kan extension along the Yoneda embedding of the composite of the forgetful functor \( \mathsf{U} \colon \atom \to \poscat \) with the simplicial nerve \( N \colon \poscat \to \sset \). 
	It has a right adjoint \( \Ex_\smallatom \colon \sset \to \atom\Set \).
\end{dfn}

\begin{prop}\label{prop:quillen_adjunction_sd_ex}
    The adjunction \( \Sd_\smallatom \dashv \Ex_\smallatom \)
    is a Quillen equivalence between the Cisinski model structure on \( \atom\Set \) and the classical model structure on \( \sSet \).
\end{prop}
\begin{proof}
Let \( \xi \colon \atom \to \Cat \) be the functor that sends any atom to the category of representable subobjects of its Yoneda embedding. 
Because atoms are regular as diagrammatic sets, the representable subobjects of an atom \( U \) are in bijection with the inclusions of atoms \( V \subseteq U \), which are in bijection with the elements of \( U \).
It follows that \( \xi U \) is a poset, isomorphic to the underlying poset of \( U \), and the diagram of functors
\[        
	\begin{tikzcd}[sep=small]
            & \poscat \\
            \atom && \sset \\
            & \Cat
            \arrow["N", from=1-2, to=2-3]
	    \arrow["\fun{U}", from=2-1, to=1-2]
            \arrow["\xi"', from=2-1, to=3-2]
            \arrow["N"', from=3-2, to=2-3]
        \end{tikzcd}
\]
    commutes up to natural isomorphism.
    As a consequence, the adjoint pair \( \Sd_\smallatom \dashv \Ex_\smallatom \) is isomorphic to the adjoint pair \( \Sd \dashv \Ex \) of \cite[\S 8.2.21]{cisinski_prefaisceaux_2006}.
	We conclude by applying \cite[Proposition 8.2.29]{cisinski_prefaisceaux_2006}.
\end{proof}

\noindent We wish to further characterise the model structure, and give an explicit set of generating acyclic cofibrations.
Given morphisms \( f \colon X \to Y \) and \( f' \colon X' \to Y' \) of diagrammatic sets, we let \( f \hat\gray f' \) denote their \emph{pushout-product}, that is, the universal morphism obtained from the diagram
\[\begin{tikzcd}
	{X \gray X'} & {X \gray Y'} \\
	{Y \gray X'} & {(X \gray Y') \cup (Y \gray X')} \\
	&& {Y \gray Y'.}
	\arrow["{X \gray f'}", from=1-1, to=1-2]
	\arrow["{f \gray X'}", from=1-1, to=2-1]
	\arrow[from=1-2, to=2-2]
	\arrow["{f \gray Y'}", curve={height=-12pt}, from=1-2, to=3-3]
	\arrow[from=2-1, to=2-2]
	\arrow["{Y \gray f'}"', curve={height=12pt}, from=2-1, to=3-3]
	\arrow["\lrcorner"{anchor=center, pos=0.125, rotate=180}, draw=none, from=2-2, to=1-1]
	\arrow["{f \hat\gray f'}"{description}, dotted, from=2-2, to=3-3]
\end{tikzcd}\]
Following \cite[\S 2.4.11]{cisinski_higher_2019}, the next definitions can be given relative to any exact cylinder in a category of presheaves, but we specialise them to diagrammatic sets for simplicity.
Let \( m \colon X \to Y \) be a monomorphism of diagrammatic sets.
For all \( \a \in \set{ +, - } \), in the pushout-product
\[        
\begin{tikzcd}
            X & Y \\
            {\thearrow{} \gray X} & {(\thearrow{} \gray X) \cup Y} \\
            && {\thearrow{} \gray Y}
            \arrow["{\iota^\a}", from=1-1, to=2-1]
            \arrow[from=2-1, to=2-2]
            \arrow["m", from=1-1, to=1-2]
            \arrow[from=1-2, to=2-2]
            \arrow["{\thearrow{} \gray m}"', curve={height=12pt}, from=2-1, to=3-3]
            \arrow["{\iota^\a}", curve={height=-12pt}, from=1-2, to=3-3]
            \arrow["{0^\a \hat\gray m}"{description}, dotted, from=2-2, to=3-3]
       	    \arrow["\lrcorner"{anchor=center, pos=0.125, rotate=180}, draw=none, from=2-2, to=1-1]
        \end{tikzcd}
\]
    the outer diagram is a pullback by Proposition \ref{prop:exact_cylinder_object} and Remark \ref{rmk:exact_cylinder_pullback}, and the legs of the pushouts are monomorphisms, so, as is the case in any category of presheaves \cite[Proposition 1.55]{johnstone1977topos}, the induced map \( 0^\a \hat\gray m \) is again a monomorphism.
    Similarly, in the pushout-product
    \[
        \begin{tikzcd}
            {\bd \thearrow{} \gray X} & {\thearrow{} \gray X} \\
            {\bd \thearrow{} \gray Y} & {\thearrow{} \gray X \cup \bd \thearrow{} \gray Y} \\
            && {\thearrow{} \gray Y}
            \arrow["{\bd \thearrow{} \gray m }"', from=1-1, to=2-1]
            \arrow["{\bd_{\thearrow{}} \gray X}", from=1-1, to=1-2]
	    \arrow["{\bd_{\thearrow{}} \gray Y}"', curve={height=18pt}, from=2-1, to=3-3]
            \arrow["{\thearrow{} \gray m}", curve={height=-18pt}, from=1-2, to=3-3]
            \arrow[from=2-1, to=2-2]
            \arrow[from=1-2, to=2-2]
	    \arrow["{\bd_{\thearrow{}} \hat\gray m}"{description}, dotted, from=2-2, to=3-3]
       	    \arrow["\lrcorner"{anchor=center, pos=0.125, rotate=180}, draw=none, from=2-2, to=1-1]
        \end{tikzcd}
\]
    the outer square is a pullback by Proposition \ref{prop:exact_cylinder_object} and stability of colimits under base change, so \( \bd_{\thearrow{}} \hat\gray m \) is also a monomorphism.
	
Given a class \( S \) of morphisms in \( \atom\Set \), we let \( l(S) \) and \( r(S) \) denote the classes of morphisms that have, respectively, the left and right lifting property with respect to all morphisms in \( S \).

\begin{dfn}[Class of anodyne extensions]
   A class of morphisms \( \An \) in \( \atom\Set \) is a \emph{class of anodyne extensions} if
   \begin{enumerate}[align=left]
	   \item[{\crtcrossreflabel{(AN0)}[enum:an0]}] there exists a set \( S \) of morphisms such that \( \An = l(r(S)) \),
	   \item[{\crtcrossreflabel{(AN1)}[enum:an1]}] for all monomorphisms \( m \), \( 0^\a \hat\gray m \) is in \( \An \), and
	   \item[{\crtcrossreflabel{(AN2)}[enum:an2]}] for all \( j \) in \( \An \), \( \bd_{\thearrow{}} \hat\gray j \) is again in \( \An \). 
   \end{enumerate}
   In this case, we call \( S \) a \emph{generating set of anodyne extensions}.
\end{dfn}

\begin{lem}\label{lem:pushout_prodcut_inclusions_rdcpx}
    For all inclusions \( m, m' \) in \( \rdcpx \), the pushout-product \( m \hat \gray m' \) is again an inclusion in \( \rdcpx \).
\end{lem}
\begin{proof}
    The fact that \( \rdcpx \) is closed under pushout-products of inclusions follows from the fact that \( \rdcpx \) has pushouts of inclusions along inclusions, and that \( i^*\colon \rdcpx \incl \atom\Set \) preserves colimits and Gray products.
	Moreover, the underlying function of \( m \hat\gray m' \) in \( \Set \) is a cartesian pushout-product of injections, which is always injective.
\end{proof}

\begin{dfn}[Horn]\label{dfn:horn}
	Let \( U \) be an atom, \( \a \in \set{+, -} \), and let \( V \submol \bd^\a U \) be a rewritable submolecule of the input or output boundary of \( U \).
    We let \( \Lambda^V_U \) be the regular directed complex \( \bd U \setminus (V \setminus \bd V) \), and call the evident inclusion \( \lambda^V_U \colon \Lambda^V_U \incl U \) a \emph{horn of \( U \)}. 
	We let
    \( J \eqdef \set{\Lambda^V_U \incl U \mid U \text{ atom, } V \submol \bd^\a U \text{ rewritable} } \).
\end{dfn}

\begin{lem}\label{lem:gray_prodcut_is_model_monoidal}
    Let \( U, V \) be atoms, and let \( W \submol \bd^\a V \) be a rewritable submolecule.
    Then
    \begin{enumerate}
        \item \( \lambda^W_V \hat\gray \bd_U = \lambda^{W \gray U}_{V \gray U} \).
        \item \( \bd_U \hat\gray \lambda^W_V = \lambda^{U \gray W}_{U \gray V} \),
    \end{enumerate}
\end{lem}
\begin{proof}
	We have that \( W \gray U \submol \bd^\a (V \gray U) \) is a rewritable submolecule by \cite[Lemma 7.2.14, Proposition 7.2.16]{amar_pasting}.
	By inspection, the image of the inclusion \( \lambda^W_V \hat\gray \bd U \) in \( V \gray U \) consists of the elements \( (v, u) \) such that \( v \in \Lambda^W_V \) and \( u \in U \), or \( v \in V \) and \( u \in \bd U \), which are precisely the elements of \( \Lambda^{W \gray U}_{V \gray U} \).
	The other equation is dual.
\end{proof}

\begin{lem}\label{lem:pushout_product_cellular_extension}
    Let \( j \colon X \to Y \) be a morphism of diagrammatic sets and let \( I, I' \) be two sets of morphisms in \( \atom\Set \). 
    If for all \( i \in I \), \( j \hat\gray i \) is in \( l(r(I')) \), then for all \( i \in l(r(I)) \), \( j \hat\gray i \) is in \( l(r(I')) \).
\end{lem}
\begin{proof}
    The functor \( - \gray - \) is biclosed, and the category \( \atom\Set \) is complete and cocomplete, so we can apply \cite[Lemma B.0.10.(i)]{henry_model} with \( I_1 \eqdef \set{j} \), \( I_2 \eqdef I \) and \( I_3 \eqdef I' \).
\end{proof}

\begin{prop}\label{prop:lr_J_is_anodyne}
    The class \( l(r(J)) \) is a class of anodyne extensions.
\end{prop}
\begin{proof}
    The category \( \rdcpx \) is small, so \( J \) is a set, hence \ref{enum:an0} is satisfied. 
    For \ref{enum:an1}, we want to show that for all monomorphisms \( m \) of diagrammatic sets, the morphism \( 0^\a \hat\gray m \) is in \( l(r(J)) \). 
    Since the class of monomorphisms is equal to \( l(r(\set{\bd_U\colon \bd U \incl U})) \), by Lemma \ref{lem:pushout_product_cellular_extension} it suffices to show that \( 0^\a \hat\gray \bd_U \) is in \( l(r(J)) \) for all atoms \( U \). 
    But the map \( 0^\a \) is a horn, so this follows from Lemma \ref{lem:gray_prodcut_is_model_monoidal}. 
	From the same lemma we deduce that \( \bd_{\thearrow{}} \hat\gray \lambda^V_U \) is in \( l(r(J)) \) for all horn inclusions \( \lambda^V_U \). 
	By Lemma \ref{lem:pushout_product_cellular_extension} again, \( \bd_{\thearrow{}} \hat\gray j \) is in \( l(r(J)) \) for all \( j \in l(r(J)) \), proving \ref{enum:an2}.
\end{proof}

\noindent In what follows, an \emph{anodyne extension} will be a morphism in \( l(r(J)) \).

\begin{prop}\label{prop:submolecule_inclusions_anodyne}
	All submolecule inclusions are anodyne extensions.
\end{prop}
\begin{proof}
	Because all isomorphisms are anodyne extensions, and anodyne extensions are closed under composition, it suffices to show that the inclusions of the form \( U \incl U \cp{k} V \) and \( V \incl U \cp{k} V \) are anodyne extensions for all \( k \geq 0 \).
	The pasting \( U \cp{k} V \) is defined by a pushout of the form
\[\begin{tikzcd}
	{\bd_k^+ U = \bd_k^- V} & V \\
	U & {U \cp{k} V}
	\arrow[hook, from=1-1, to=1-2]
	\arrow[hook', from=1-1, to=2-1]
	\arrow[hook', from=1-2, to=2-2]
	\arrow[hook, from=2-1, to=2-2]
	\arrow["\lrcorner"{anchor=center, pos=0.125, rotate=180}, draw=none, from=2-2, to=1-1]
\end{tikzcd}\]
and since anodyne extensions are closed under pushouts, it suffices to prove that \( \bd_k^\a U \incl U \) is an anodyne extension for all \( k \geq 0, \a \in \set{ +, - } \).
	This is trivial when \( U \) is a point, so we may proceed by induction on submolecules \cite[Comment 4.1.7]{amar_pasting}, and assume that all submolecule inclusions into proper submolecules of \( U \) are anodyne extensions.
	Suppose \( U \) is an atom, \( n \eqdef \dim U \).
	For \( k \geq n \), the inclusion is an identity.
	For \( k = n - 1 \), then \( \bd_k^\a U = \bd^\a U \) is equal to the horn \( \Lambda_U^{\bd^{-a}U} \), whose inclusion is an anodyne extension by definition.
	For \( k < n - 1 \), then by globularity the inclusion factors through \( \bd_k^\a U \incl \bd^\a U \), and we can apply the inductive hypothesis.

	Finally, suppose \( U \) splits into proper submolecules \( V \cp{i} W \).
	If \( k \le i \), then the inclusion of \( \bd_k^\a U \) factors through a submolecule inclusion into \( V \) or \( W \), and we can apply the inductive hypothesis.
	If \( k > i \), then the squares
	\[
\begin{tikzcd}
            {\bd^\a_k V} & {\bd^\a_k V \cp i \bd^\a_k W} && {\bd^\a_k W} & {V \cp i \bd^\a_k W} \\
            V & {V \cp i \bd^\a_k W} && W & {V \cp i W}
            \arrow[hook, from=1-1, to=1-2]
            \arrow[hook', from=1-1, to=2-1]
            \arrow["{j_1}", hook', from=1-2, to=2-2]
            \arrow[hook, from=1-4, to=1-5]
            \arrow[hook', from=1-4, to=2-4]
            \arrow["{j_2}", hook', from=1-5, to=2-5]
            \arrow[hook, from=2-1, to=2-2]
            \arrow[hook, from=2-4, to=2-5]
	    \arrow["\lrcorner"{anchor=center, pos=0.125, rotate=180}, draw=none, from=2-2, to=1-1]
    	\arrow["\lrcorner"{anchor=center, pos=0.125, rotate=180}, draw=none, from=2-5, to=1-4]
        \end{tikzcd}
\]
	are pushouts in \( \rdcpx \).
	By the inductive hypothesis, \(\bd_k^\a V \incl V \) and \(\bd_k^\a W \incl W \) are anodyne extensions, hence so are \( j_1, j_2 \) and their composite, which is equal to the inclusion \( \bd_k^\a U \incl U \).
\end{proof}

\begin{cor}\label{cor:any_mono_atom_anodyne}
    All inclusions of atoms are anodyne extensions.
\end{cor}
\begin{proof}
    Follows from Proposition \ref{prop:submolecule_inclusions_anodyne} and \cite[Lemma 4.1.5]{amar_pasting}.
\end{proof}

\begin{lem}\label{lem:horn_contractible}
    Let \( U \) be an atom and \( \Lambda^V_U \) a horn of \( U \).
    Then the horn inclusion \( \lambda^V_U\colon \Lambda^V_U \incl U \) is an \( \infty \)-equivalence.
\end{lem}
\begin{proof}
    The functor \( \Sd_\smallatom \) coincides with the functor \( -^\Delta \) from \cite[\S 10.3.1]{amar_pasting}.
    Since \( \Sd_\smallatom \) is left adjoint, the pushout diagram
    \[
        \begin{tikzcd}
            {\bd V} & {\Lambda^V_U} \\
            V & \bd U
            \arrow[hook, from=1-1, to=1-2]
            \arrow[hook', from=1-1, to=2-1]
            \arrow[hook', from=1-2, to=2-2]
            \arrow[hook, from=2-1, to=2-2]
	    \arrow["\lrcorner"{anchor=center, pos=0.125, rotate=180}, draw=none, from=2-2, to=1-1]
        \end{tikzcd}
\]
    of regular directed complexes is sent to a pushout diagram in \( \sset \).
	By \cite[Proposition 10.1.29.(2)]{amar_pasting} and \cite[Proposition 10.3.2]{amar_pasting}, we conclude that \( \Sd_\smallatom(\Lambda^V_U) \) is a PL-ball, thus it is contractible. 
	It follows that \( \Sd_\smallatom(\lambda^V_U) \) is a weak equivalence in \( \sset \), as is any map between contractible objects. 
	Since left Quillen equivalences reflect weak equivalences between cofibrant objects \cite[Corollary 1.3.16]{hovey_model_2007}, and every object is cofibrant in the classical model structure on \( \sset \), we conclude that \( \lambda^V_U \) is an \( \infty \)\nbd equivalence.
\end{proof}

\begin{thm}\label{thm:main_theorem}
	In the Cisinski model structure on \( \atom\Set \),
\begin{enumerate}
	    \item the cofibrations are generated by the inclusions \( \bd U \incl U \),
	    \item the acyclic cofibrations are generated by the horns \( \Lambda^V_U \incl U \),
\end{enumerate}
	where \( U \) ranges over all atoms.
\end{thm}

\begin{proof}
    By Theorem \ref{thm:cisinski_model_structure}, the cofibrations are the monomorphisms, which are generated by the inclusions \( \bd U \incl U \), where \( U \) ranges over all atoms.

	Let \( J \) be the set of horn inclusions.
	By \cite[Proposition 1.3.13]{cisinski_prefaisceaux_2006} applied to \( \mathfrak{J} \eqdef (\thearrow{} \gray -, (\iota^-, \iota^+), \sigma) \) and \( \mathscr{M} \eqdef \set{ \bd U \incl U } \), there exists a smallest class \( \Lambda_{\mathfrak{J}}(J, \mathscr M) \) of anodyne extensions containing \( l(r(J)) \), which by Proposition \ref{prop:lr_J_is_anodyne} is actually equal to \( l(r(J)) \).
    	The statement then follows from \cite[Corollaire 8.2.19]{cisinski_prefaisceaux_2006},
    	whose assumption (\textit{a}) is satisfied by Lemma \ref{lem:horn_contractible}, assumption (\textit{b}) is proved in Lemma \ref{cor:any_mono_atom_anodyne}, and assumption (\textit{c}) is true because \( \thearrow{} \gray U \) is representable for all atoms \( U \). 
\end{proof}

\begin{thm}\label{thm:model_structure_monoidal_Gray}
	The Cisinski model structure on \( \atom\Set \) is monoidal with respect to the Gray product, whose derived functor is the cartesian product in the homotopy category.
\end{thm}
\begin{proof}
	The pushout-product axiom is given by Lemma \ref{lem:pushout_prodcut_inclusions_rdcpx} for the generating cofibrations and Lemma \ref{lem:gray_prodcut_is_model_monoidal} for the generating acyclic cofibrations.
	The unit axiom is satisfied since all objects are cofibrant.
	This proves that the model structure is monoidal.

	If \( U, V \) are atoms, since the underlying poset of \( U \gray V \) is the cartesian product, and the nerve functor \( N\colon \poscat \to \sset \) is right adjoint, we have a natural isomorphism
    \begin{equation*}
	    \Sd_\smallatom(U \gray V) = N(U \times V) \cong NU \times NV = \Sd_\smallatom(U) \times \Sd_\smallatom(V).
    \end{equation*}
    Since \( \Sd_\smallatom \) is left adjoint, and \( - \gray - \) and \( - \times - \) are closed respectively in \( \atom\Set \) and \( \sset \), the above isomorphism extends along colimits to all diagrammatic sets, that is, \( \Sd_\smallatom \) sends Gray products to cartesian products in \( \sset \), whose derived functor is the cartesian product.
    We conclude by Proposition \ref{prop:quillen_adjunction_sd_ex}.
\end{proof}

\noindent Although it is clear that the good notion of product of diagrammatic sets is the Gray product, nevertheless we conclude this paper by showing that \( \atom \) is a \emph{strict} test category. 
We shortcut the study of the impractical cartesian product of diagrammatic sets by factorising the Quillen equivalence
\[        \begin{tikzcd}
	\sset && \sset 
            \arrow[""{name=0, anchor=center, inner sep=0}, "{\Sd}", curve={height=-12pt}, from=1-1, to=1-3]
            \arrow[""{name=1, anchor=center, inner sep=0}, "{\Ex}", curve={height=-12pt}, from=1-3, to=1-1]
            \arrow["\dashv"{anchor=center, rotate=-90}, draw=none, from=0, to=1]
        \end{tikzcd}
\]
through \( \atom\Set \), using the above Quillen equivalence \( \Sd_\smallatom \dashv \Ex_\smallatom \) and another Quillen adjunction we define below. 
With the right adjoint of the latter, we can send products of diagrammatic sets to products of simplicial sets, whose homotopical properties are already well-understood.

\begin{comm}
	The fact that \( \atom \) is a strict test category may, at first, come as a surprise: taking \( \square \) to be the full subcategory of \( \atom \) on iterated Gray products of \( \thearrow{} \), we obtain a cubical category \emph{without} connections (the coconnections maps are not cartesian, see Remark \ref{rmk:cartesian_maps}).
This is famously a test category but not a strict test category \cite{Maltsiniotis2009}: for instance, the product of two intervals in the cubical category has the homotopy type of \( S^2 \vee  S^1 \).
However, the simplex category \( \Delta \) is a strict test category, and since we also recover it as a full subcategory of \( \atom \), we should rather expect that we have enough maps to make the cartesian product of diagrammatic sets homotopically well behaved.
\end{comm}

\noindent The full subcategory inclusion \( i\colon \simplexcat \incl \atom \) induces by left Kan extension along the Yoneda embedding a functor \( i_\Delta \colon \sset \to \atom\Set \), together with its right adjoint \( (-)_\Delta \colon \atom\Set \to \sset \). 
Since \( i \) is full and faithful, so is \( i_\Delta \). 
We therefore have the following adjunctions:
    \[
        \begin{tikzcd}
		\sset && \atom\Set && \sset. 
            \arrow[""{name=0, anchor=center, inner sep=0}, "{i_\Delta}", curve={height=-12pt}, from=1-1, to=1-3]
            \arrow[""{name=1, anchor=center, inner sep=0}, "{(-)_\Delta}", curve={height=-12pt}, from=1-3, to=1-1]
            \arrow[""{name=2, anchor=center, inner sep=0}, "{\Sd_\smallatom}", curve={height=-12pt}, from=1-3, to=1-5]
            \arrow[""{name=3, anchor=center, inner sep=0}, "{\Ex_\smallatom}", curve={height=-12pt}, from=1-5, to=1-3]
            \arrow["\dashv"{anchor=center, rotate=-90}, draw=none, from=0, to=1]
            \arrow["\dashv"{anchor=center, rotate=-90}, draw=none, from=2, to=3]
        \end{tikzcd}
\]
    We claim that \( \Sd_\smallatom i_\Delta \) is naturally isomorphic to the barycentric subdivision endofunctor \( \Sd \), hence \( (\Ex_\smallatom -)_\Delta \) is naturally isomorphic to its right adjoint \( \Ex \) \cite[Section 4.6]{fritsch1990cellular}. 
    The restriction of the forgetful functor \( \atom \to \poscat \) to \( \Delta \) is precisely the functor sending the \( n \)\nbd simplex to its poset of non-degenerate simplices ordered by inclusion.
    By definition, its post-composition with \( N \) is the barycentric subdivision \( \Sd \colon \Delta \to \sset \). 
    Then the diagram of functors
    \[
        \begin{tikzcd}
            \sset & \atom\Set \\
            \Delta & \atom & \poscat & \sset
            \arrow["{i_\Delta}", from=1-1, to=1-2]
            \arrow["{\Sd_\smallatom}", curve={height=-12pt}, from=1-2, to=2-4]
            \arrow[hook, from=2-1, to=1-1]
            \arrow[hook, from=2-1, to=2-2]
            \arrow["\Sd"', curve={height=18pt}, from=2-1, to=2-4]
            \arrow[hook, from=2-2, to=1-2]
	    \arrow["\fun{U}", from=2-2, to=2-3]
            \arrow["N", from=2-3, to=2-4]
        \end{tikzcd}
\]
    commutes up to natural isomorphism. 
    Because \( i_\Delta \) is the left Kan extension of \( \Delta \incl \atom\Set \) along \( \Delta \incl \sset \), and \( \Sd_\smallatom \), as a left adjoint, preserves left Kan extensions, \( \Sd_\smallatom i_\Delta \) is the left Kan extension of \( \Sd \colon \Delta \to \sset \) along the Yoneda embedding, which is by definition \( \Sd \) up to natural isomorphism. 

\begin{prop}\label{prop:other_adjunction_is_quillen}
    The adjunction \( i_\Delta \dashv (-)_\Delta \) is a Quillen equivalence between the classical model structure on \(\sset \) and the Cisinski model structure on \( \atom\Set \).
\end{prop}

\begin{proof}
    The boundary inclusions \( \bd \Delta^n \incl \Delta^n \) and the horns \( \Lambda^k_n \incl \Delta^n \) in \( \sset \) are sent via \( i_\Delta \) to boundary inclusions and horns in \( \atom\Set \), therefore the adjunction is Quillen.
    Since \( \Sd \dashv \Ex \) is a Quillen auto-equivalence by \cite[Proposition 8.2.29]{cisinski_prefaisceaux_2006} and \( \Sd_\smallatom \dashv \Ex_\smallatom \) is a Quillen equivalence by Proposition 
    \ref{prop:quillen_adjunction_sd_ex}, by the 2-out-of-3 property \( i_\Delta \dashv (-)_\Delta \) is also a Quillen equivalence.
\end{proof}

\begin{prop}\label{prop:map_gray_to_cart_is_weak_eq}
    The category \( \atom \) is a strict test category. 
\end{prop}
\begin{proof}
    By Proposition \ref{prop:other_adjunction_is_quillen}, for all diagrammatic sets, \( \Sd (X_\Delta) \) is weakly equivalent to \( \Sd_\smallatom X \) in \( \sset \). 
    Moreover, since \( \Delta \) is a strict test category, \( \Sd \) preserves cartesian products up to weak equivalence. If \( X, Y \) are diagrammatic sets, letting \( \simeq \) denote weak equivalence, we have
    \begin{align*}
        \Sd_\smallatom(X) \times \Sd_\smallatom(Y) &\simeq \Sd (X_\Delta) \times \Sd (Y_\Delta) \\
        &\simeq \Sd(X_\Delta \times Y_\Delta) \\
        &\cong \Sd((X \times Y)_\Delta) && ((-)_\Delta \text{ is right adjoint}) \\
        &\simeq \Sd_\smallatom (X \times Y).
    \end{align*}
    This proves that \( \Sd_\smallatom \) preserves finite products up to weak equivalence.
    It follows that \( \atom \) is a strict test category. 
\end{proof}

\begin{cor}\label{cor:gray_and_times_weakly_equivalent}
	Let \( X, Y \) be diagrammatic sets.
	Then \( X \gray Y \) and \( X \times Y \) are weakly equivalent in the Cisinski model structure.
\end{cor}
\begin{proof}
	By Theorem \ref{thm:model_structure_monoidal_Gray} and Proposition \ref{prop:map_gray_to_cart_is_weak_eq}, both the Gray product and the cartesian product induce the cartesian product in the homotopy category, which implies that the canonical map \( X \gray Y \to X \times Y \) is a weak equivalence.
\end{proof}

\bibliographystyle{alpha}
\small \bibliography{main.bib}

\end{document}